\documentclass[12 pts]{article}

\usepackage{amsmath, amsthm, amssymb,amsfonts,latexsym,cancel}
\usepackage{color}
\usepackage{array}
\usepackage[bottom]{footmisc}
\usepackage[table]{xcolor}
\usepackage{ifsym}
\usepackage{float}
\usepackage[all]{xy}
\usepackage{ae,aecompl}
\usepackage[pdftex]{graphicx}
\DeclareGraphicsExtensions{.png,.pdf,.jpg,.bmp}
\usepackage{epsfig}
\usepackage{algorithm}
\usepackage{listings}
\usepackage{multirow}
\usepackage{rotating}
\usepackage{booktabs}
\usepackage{fancyhdr}
\usepackage[pdftex=true,colorlinks=false,plainpages=false]{hyperref}
\usepackage[margin=1in]{geometry}
\usepackage{textcomp}
\usepackage[T1]{fontenc}
\usepackage{relsize}
\usepackage{makecell}
\usepackage[backend=bibtex, sorting=nty, bibstyle=ieee]{biblatex}  

\usepackage{mathtools}

\usepackage[justification=centering]{caption}
\usepackage{subcaption} 
 
\usepackage{titlesec}
\titlespacing{\paragraph}{%
  0pt}{
  0.0\baselineskip}{
  1em}

\usepackage{amsmath,amsthm,amssymb}
\usepackage{graphicx}
\usepackage{epsfig}
\usepackage{multirow}
\usepackage{caption}
\usepackage{subcaption}

\usepackage{authblk}

\newtheorem{lemma}{Lemma}

\newtheorem{proposition}{Proposition} 
 
\usepackage{booktabs}
\usepackage{longtable}
\usepackage{adjustbox}
\usepackage{ltablex,array}
\usepackage{tabu}

\newcommand{\R}{\mathbb{R}}

\usepackage[noend]{algpseudocode}

\makeatletter
\def\BState{\State\hskip-\ALG@thistlm}
\makeatother

\DeclareMathOperator*{\argmin}{Argmin}

\usepackage[titletoc,title]{appendix}

\bibliography{References} 
 
\begin{document}

\title{\vspace{-2em}A Simulated Annealing Algorithm for the Directed Steiner Tree Problem}
\author[1]{Matias Siebert}
\author[1]{Shabbir Ahmed}
\author[1]{George Nemhauser}
\affil[1]{H. Milton Stewart School of Industrial and Systems Engineering, Georgia Institute of Technology, Atlanta, GA}
\date{}
\maketitle

\begin{abstract}
\noindent In \cite{siebert2019linear} the authors present a set of integer programs (IPs) for the Steiner tree problem, which can be used for both, the directed and the undirected setting of the problem. Each IP finds an optimal Steiner tree with a specific structure. A solution with the lowest cost, corresponds to an optimal solution to the entire problem. The authors show that the linear programming relaxation of each IP is integral and, also, that each IP is polynomial in the size of the instance, consequently, they can be solved in polynomial time. The main issue is that the number of IPs to solve grows exponentially with the number of terminal nodes, which makes this approach impractical for large instances.

In this paper, we propose a local search procedure to solve the directed Steiner tree problem using the approach presented in \cite{siebert2019linear}. In order to do this, we present a dynamic programming algorithm to solve each IP efficiently. Then we provide a characterization of the neighborhood of each tree structure. Finally, we use the proposed algorithm and the neighborhood characterization to solve the problem using a simulated annealing framework. Computational experiments show that the quality of the solutions delivered by our approach is better than the ones presented in the literature for the directed Steiner tree problem.
\end{abstract}

\section{Introduction}

The Steiner tree problem is a classic network design problem. There are two versions, directed and undirected. In the directed version (DST), we are given a directed graph $D=(V,A)$, with non-negative arc costs $c_a$ for all $a\in A$, a root node $r\in V$, and a set of terminal nodes $R\subset V$. In the undirected case (UST), we are given an undirected graph $G=(V,E)$, with non-negative edge costs $c_e$ for all $e\in E$, and a set of terminal nodes $R\subseteq V$. In the latter case, one can always take the bidirected version of the graph $G$, pick an arbitrary terminal node $r\in R$ as a root node, and solve the problem as a directed Steiner tree. Note that we may assume without loss of generality that $c > 0$, since we can contract all the edges with cost 0.

The undirected version of the problem is known to be an NP-Hard \cite{karp1972reducibility}, and it is even NP-Hard to find an approximate solution whose cost is within a factor of $\frac{96}{95}$ of the optimum \cite{bern1989steiner,chlebik2008steiner}. Nevertheless, there are constant factor approximation algorithms for this problem \cite{byrka2013steiner,robins2005tighter,zelikovsky199311}, with $\ln(4)+\epsilon$ being the best known approximation factor for the undirected problem.

On the other hand, for DST there are not any known constant factor approximation algorithms. Actually, in \cite{halperin2003polylogarithmic}, the authors show that the directed Steiner tree problem admits no $\mathcal{O}\left(\log^{2-\epsilon}(|R|)\right)$, unless $NP\subseteq ZTIME(n^{polylog(n)})$. The best known approximation ratio for the DST problem is $\mathcal{O}\left(|R|^{\frac{1}{t}}\right)$ for $t>1$, which corresponds to the algorithm presented in \cite{charikar1999approximation}. This algorithm constructs low density directed Steiner trees in polynomial time, where the density of a tree is defined as the ratio of the cost of the tree over the number of terminals of the tree.

For UST, there are several efficient algorithms to reduce the instances size, keeping at least one optimal solution. Moreover, there exists efficient algorithms to solve the problem \cite{duin1997efficient,duin2000preprocessing,polzin2001improved,polzin2002extending}. These algorithms solve to optimality a large part of the problems in the SteinLib library \cite{koch2001steinlib}. Since all of those algorithms assume the graph is undirected, in most cases they cannot be extended to the directed problem \cite{rehfeldt2015generic}. 

In contrast, there are not many papers in the literature related to computational experiments for the directed Steiner tree problem. In \cite{watel2016practical}, the authors present two approximation algorithms, with $\mathcal{O}(|R|)$ and $\mathcal{O}(\sqrt{|R|})$ approximation ratios respectively. The authors present computational experiments, creating directed instances from over 900 undirected instances of the SteinLib library. All the directed graphs created by the authors keep at least one optimal solution of the original instance, consequently, they know in advance the value of an optimal solution. They compare their proposed algorithms against four benchmark algorithms used in practice. The authors conclude that their proposed algorithms outperform the rest of the studied algorithms in terms of solution quality, while having similar execution times.

In this paper we focus on the new approach proposed in \cite{siebert2019linear}, where the authors present a set of integer programs (IP) for the Steiner tree problem. The solution of an IP with lowest cost, corresponds to an optimal solution of the problem. Each IP corresponds to a specific tree structure, and its size is polynomial in the size of the instance. Moreover, the linear relaxation of each IP is integral, therefore each sub-problem can be solved in polynomial time. The issue with the proposed approach is that the number of IPs to solve grows exponentially with the number of terminal nodes. We present a dynamic programming algorithm and a simulated annealing based framework to address this issue, and we present computational experiments comparing the proposed approach to all the algorithms studied in \cite{watel2016practical}.

The remainder of the paper is organized as follows. Section \ref{Sec:Description_LP_Based_Approach} provides a detailed review of the approach presented in \cite{siebert2019linear}, and introduces definitions used in the paper. Section \ref{Sec:Algorithm_Section} presents an efficient algorithm to solve each subproblem of the approach presented in \cite{siebert2019linear}. Section \ref{Sec:Neighborhood_Characterization} provides a complete characterization of the neighborhood of each subproblem, which is needed since simulated annealing performs a local search at each iteration. Section \ref{Sec:Simulated_Annealing} presents the proposed simulated annealing framework to solve the problem. This section also provides a routine to check and, potentially, improve the quality of the solutions obtained at each iteration of the simulated annealing algorithm. Moreover, this section also presents a special analysis for the case of rectilinear graphs. In Section \ref{Sec:Results}, we present computational experiments obtained by solving the directed Steiner tree problem using our proposed formulation. Computational experiments show that the quality of the solutions delivered by our approach is better than the ones presented in the literature for the directed Steiner tree problem. Section \ref{Sec:Conclusions} gives conclusions.

\section{Description of LP-based approach}\label{Sec:Description_LP_Based_Approach}

The directed Steiner tree problem is to construct a minimum cost directed tree, rooted at $r$, which spans all terminal nodes. In \cite{siebert2019linear}, the authors give a new approach to solve this problem. It involves solving a set of independent IPs, with the property that the solution of an IP with lowest cost corresponds to an optimal solution of the problem. Each IP solves the problem for a fixed Steiner tree structure, where each structure is defined by the way the paths from $r$ to each terminal node share arcs. In this approach, one commodity per terminal node is created. The source node of all commodities is the root node $r$, and the sink is the corresponding terminal node. Figure \ref{Fig:Steiner_the_Structure_Example} illustrates the idea.

\begin{figure}[H]
\begin{subfigure}{.5\textwidth}
  \centering
 \includegraphics[width=0.8\textwidth]{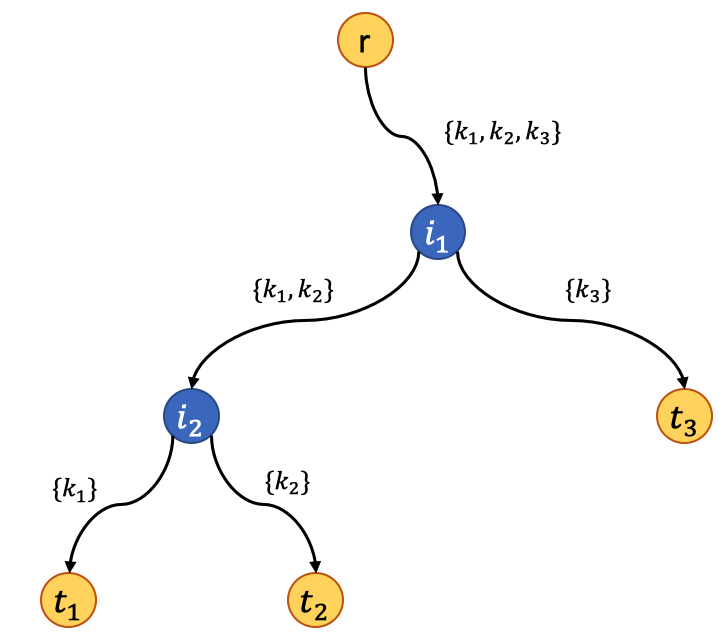}
  \caption{3 terminals directed Steiner tree rooted at $r$.}
  \label{Subfig:Example_Directed_Steiner_tree}
\end{subfigure}
\begin{subfigure}{.5\textwidth}
  \centering
 \includegraphics[width=0.738\textwidth]{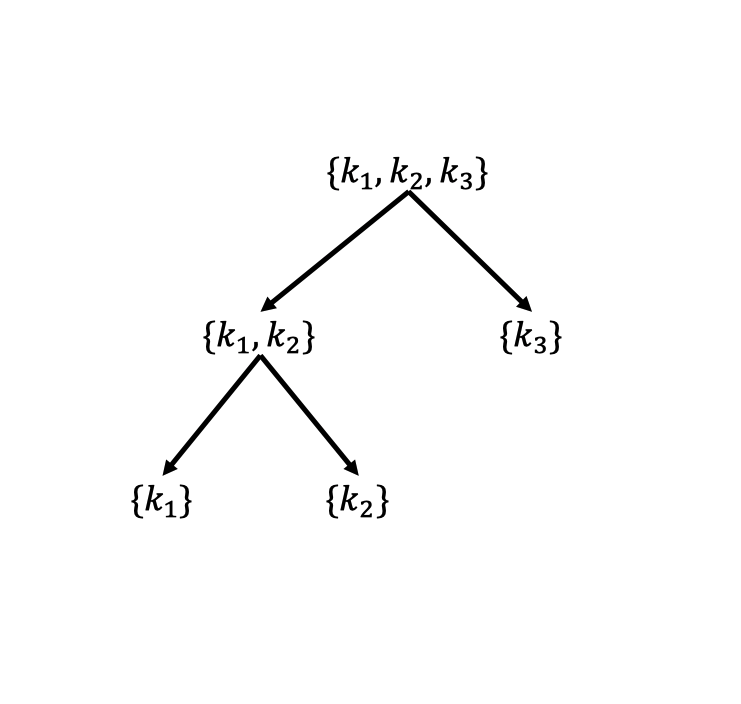}
  \caption{Structure of directed Steiner tree.}
  \label{Subfig:Structure_of_Seiner_tree}
\end{subfigure}
\caption{A 3 terminals directed Steiner tree rooted at $r$, and tree representation of its laminar structure. }
\label{Fig:Steiner_the_Structure_Example}
\end{figure}

Figure \ref{Subfig:Example_Directed_Steiner_tree} shows a directed Steiner tree, which is composed of 5 paths. Above each path, the set of commodities that share arcs in such paths is highlighted. Figure \ref{Subfig:Structure_of_Seiner_tree} displays the tree representation of the laminar structure of the sets that are present in the solution of the tree shown in Figure \ref{Subfig:Example_Directed_Steiner_tree}.

The main result in \cite{siebert2019linear} is that the formulation for each structure is perfect, meaning that the linear programming (LP) relaxation of each IP corresponds to the convex hull of the set of all feasible solutions to the IP. Consequently, each IP can be solved by solving its LP relaxation. Furthermore, each IP has polynomial size with respect to the size of the input data, consequently, they can be solved efficiently, even for large instances. The main drawback of the proposed approach, is that the number of IP problems to solve grows exponentially with the number of terminal nodes, which makes this approach tractable only when the number of terminal nodes is fixed, and impractical for real-world size instances. For a detailed explanation of this approach we refer the readers to \cite{siebert2019linear}.

\subsection{Notation}

In this section, we introduce notation used throughout this paper. We assume we have a directed Steiner tree instance. Let $D=(V,A)$ be the directed input graph, where $V$ is the set of nodes of the graph, and $A$ is the set of arcs of the graph. We are also given a root node $r\in V$, and a set of terminal nodes $R\subseteq V\backslash\{r\}$. The nodes in  $V\backslash(R\cup\{r\})$ are called Steiner nodes. For each terminal node $t_k\in R$, we define a commodity $k$, and we define $K$ to be the set of all commodities. All commodities share the same source node, which is $r$, and for each commodity $k\in k$, its sink node is the corresponding terminal $t_k\in R$. Let $b$ be the number of commodities, and let $S$ be the set of all non-empty subsets of elements in $K$.

As discussed in section \ref{Sec:Description_LP_Based_Approach}, the tree structure of every directed Steiner tree is defined by the sets of commodities present in such tree. This collection of sets of $S$ forms a laminar family. Recall that a family $\mathcal{C}$ of sets is {\it laminar} if for every $A,B\in \mathcal{C}$ we either have that $A\subseteq B$, or $B\subseteq A$, or $A\cap B=\emptyset$. We define $\mathcal{L}_b$ to be the set of admissible laminar families that describe the structure of a Steiner tree with $b$ commodities, where an admissible laminar family is one that contains the set of all commodities, and all the singletons. For laminar family $l\in\mathcal{L}_b$, we define $S(l)$ as the set of subsets of $K$ that are present in laminar family $l$.

We say that $\hat{s}\in S(l)$ is the {\it parent} set of $s'\in S(l)$, if $s'$ is one of the sets obtained when $\hat{s}$ is split according to laminar family $l$. Moreover, we say that $s'$ is a {\it child} set of $\hat{s}$ in $l$. Notice that, for all $s'\in S(l)$ with $|s'|<|K|$, there is exactly one parent set of $\hat{s}$ in $S(l)$. Also, for each set $\hat{s}\in S(l)$ with $|\hat{s}|\geq 2$, there is at least two child sets of $\hat{s}$ in $S(l)$. We say that a directed Steiner tree $T$ in $D$ follows an $(r,l)$ {\it structure}, if $T$ is an arborescence rooted in $r$ that follows the splitting sequence defined by laminar family $l\in\mathcal{L}_b$.

Finally, we define $\mathcal{Z}_l$ to be the sub-problem of laminar family $l\in\mathcal{L}_b$, where we look for a lowest cost $(r,l)$ structured Steiner tree.

\section{Algorithm to solve $\mathcal{Z}_l$}\label{Sec:Algorithm_Section}

Although $\mathcal{Z}_l$ can be solved by a linear programming algorithm as noted above, we will present in this section a recursive algorithm that is much more efficient.

\subsection{Observations and algorithm intuition}

Consider a fixed laminar family $l\in\mathcal{L}_b$, and let $s\in S(l)$ be such that $2\leq |s|<|K|$. Since we are in a particular laminar family $l$, then we know the way $s$, and its subsets, split. Now, suppose we assume that the path of set $s$ starts at node $i$, then the solution is equivalent to finding a minimum Steiner tree rooted at $i$, with $|s|$ commodities\footnote{This is equivalent to a directed Steiner tree whose root node is $i$, and that has $|s|$ terminal nodes which correspond to $\{t_k\}_{k\in s}$}, that follows the splitting sequence for $s$ defined in $l$. We will denote $l(s)$, to the ``sub-laminar'' family of $l$, when we only focus on $s$.

For any laminar family, the problem reduces to finding the splitting nodes, because once we have the splitting nodes, then we only need to connect the corresponding nodes by the shortest paths between them. The proposed algorithm uses this fact, and since we do not know the splitting nodes in advance, then we need to compute the shortest path between all pairs of nodes in $D$.

Consider the following example. Suppose we want to solve a directed Steiner tree with 4 terminals. Also, suppose we are solving the sub-problem for the laminar family shown in Figure \ref{Fig:Sol1}

\begin{figure}[H]
\begin{center}
\includegraphics*[width=0.6\textwidth]{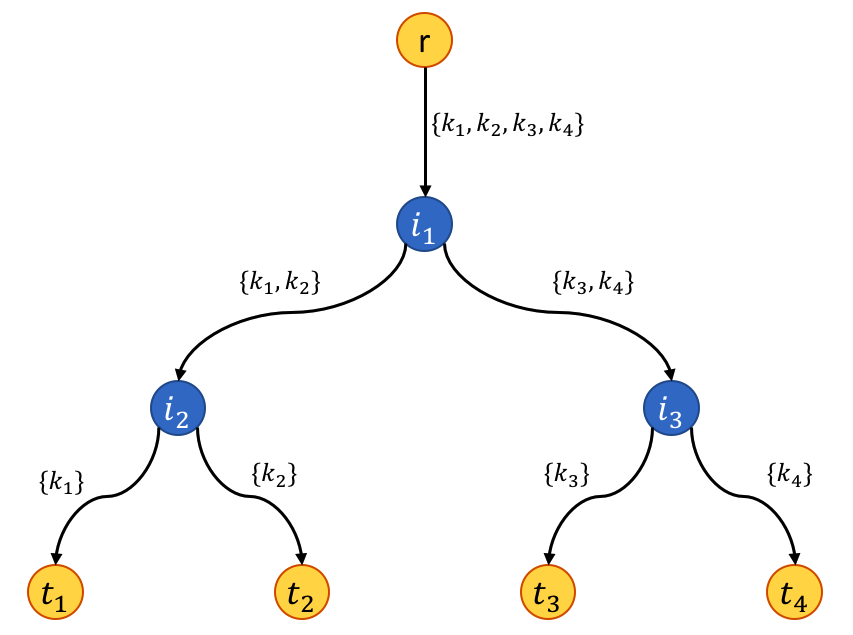}
\caption{Laminar family representation.}
\label{Fig:Sol1}
\end{center}
\end{figure}

We have to find nodes $i_1$, $i_2$, and $i_3$, shown in Figure \ref{Fig:Sol1}. The idea of the proposed algorithm to solve $\mathcal{Z}_l$ is the following. For all sets $s\in S(l)$ with $2\leq |s|<|K|$, and for all nodes $i\in V$, we define the minimum $(i,l(s))$ structured Steiner tree, whose optimal solution will be denoted by $x(i,s)$. Now, consider the set $\{k_1,k_2\}$ in the example shown in Figure \ref{Fig:Sol1}. If we fix a given node $i$ as a root for its corresponding Steiner tree, then we can compute $x(i,\{k_1,k_2\})$ in the following fashion. For all $j\in N$, we compute the sum of 3 shortest paths, which are an $i$-$j$ shortest path, a $j$-$t_1$ shortest path, and a $j$-$t_2$ shortest path. Let $j^*$ be a node such that the sum is the smallest, then $x(i,s)$ is the union of an $i$-$j^*$ shortest path, a $j^*$-$t_1$ shortest path, and a $j^*$-$t_2$ shortest path. We repeat that procedure for all $i\in V$ and we will have computed all possible Steiner trees for $\{k_1,k_2\}$. This procedure is the same for $\{k_3,k_4\}$. For set $\{k_1,k_2,k_3,k_4\}$, the process is similar, since $x(i,\{k_1,k_2,k_3,k_4\})$ is going to be the union of 3 solutions, but in this case, it is the union of a shortest $i$-$j^*$ path, $x(j^*,\{k_1,k_2\})$ and $x(j^*,\{k_3,k_4\})$. Note, that since in this case $K=\{k_1,k_2,k_3,k_4\}$, then we only care about $x(r,\{k_1,k_2,k_3,k_4\})$.

\subsection{Proposed algorithm}

Before stating the algorithm, we introduce the notation used within this section.
\begin{itemize}
\item $sp(i,j)$: Collection of arcs that compose a shortest $i$-$j$ path in $D$.
\item $c(sp(i,j))$: Cost of a shortest $i$-$j$ path in $D$.
\item $z(i,s)$: Cost of solution of set $s$ rooted in $i$.
\item $j^*_s(i)$: Node at which set $s$ splits, for an optimal $(i,l(s))$ structured directed Steiner tree.
\item $N(s)$: Nodes that can be root node for subset $s$. Note that $N(s)=N$ for all $s\in S(l)$, with $s\subset K$. For $s=K$ we have $N(s)=r$.
\end{itemize}

Observe that, for all $k\in K$, $z(i,k)=c(sp(i,t_k))$. Since we only focus on laminar families whose tree representation is a full binary tree (see Proposition 5 in \cite{siebert2019linear}), then each set with at least 2 elements has exactly 2 children. We denote $s_1$ and $s_2$ to be the child sets of set $s\in S(l)$, $|s|\geq 2$. As an abuse of notation, we define $x(i,s)$ as follows $x(i,s)=sp(i,j)+x(j,s_1)+x(j,s_2)$, which means that the solution of the directed Steiner tree, rooted in $i$ using sets in $s$, is the union of a shortest $i$-$j$ path, and the solutions of its two child sets $s_1$ and $s_2$ rooted in $j$, where $j$ is the splitting node of set $s$. 

The recursion to compute $z(i,s)$ and $x(i,s)$ is the following,
\begin{align*}
z(i,s) &= \min_{j\in N}\Big\{c(sp(i,j))+z(j,s_1)+z(j,s_2)\Big\} & \forall i\in N(s), s\in S(l): |s|\geq 2\\
j^*_s(i) &\in \argmin_{j\in N}\Big\{c(sp(i,j))+z(j,s_1)+z(j,s_2)\Big\} & \forall i\in N(s), s\in S(l): |s|\geq 2\\
x(i,s) &= sp(i,j^*_s(i))+x(j^*_s(i),s_1)+x(j^*_s(i),s_2)& \forall i\in N(s), s\in S(l): |s|\geq 2
\end{align*}
Note that for $s=K$, we only need to compute $z(r,K)$ and $j^*_K(r)$. Assume that $S(l)$ is ordered in increasing order of the cardinality of its elements. Then, the algorithm to solve $\mathcal{Z}_l$ is the following,

\begin{algorithm}[H]
\caption{Compute $x(i,s)$ and $z(i,s)$ for $s\in S(l), |s|\geq 2$}\label{Alg:DP_for_Z_l}
\begin{algorithmic}[1]
\State Set $x(i,k)=sp(i,t_k)$ for all $k\in K,i\in V$.
\State Set $z(i,k)=c(sp(i,t_k))$ for all $k\in K,i\in V$.
\State Set $z(i,s)=+\infty$ for all $s\in S(l), |s|\geq 2,i\in N$
\For {$s\in S(l), |s|\geq 2$}
\For {$i\in N(s)$}
\For {$j\in N$}
\If {$c(sp(i,j))+z(j,s_1)+z(j,s_2)<z(i,s)$}
\State $x(i,s) \leftarrow sp(i,j)+x(j,s_1)+x(j,s_2)$
\State $z(i,s) \leftarrow c(sp(i,j))+z(j,s_1)+z(j,s_2)$
\EndIf
\EndFor
\EndFor
\EndFor
\Return $x(r,K)$ and $z(r,K)$.
\end{algorithmic}
\end{algorithm}

First, note that since the elements in $S(l)$ are ordered in increasing order of their cardinality, then for all $s\in S(l)$ with $|s|\geq 2$ we compute $x(i,s_1)$ and $x(i,s_2)$ before computing $x(i,s)$ and its cost $z(i,s)$, since $|s_1|<|s|$ and $|s_2|<|s|$. Now, the {\it for loop} in step 5 is looping over all possible root nodes for set $s$. The {\it for loop} in step 6, loops over all splitting nodes for set $s$. In step 7, we check whether the current best solution can be improved. If the solution can be improved, then the best solution is updated. Finally, the solution is given by $x(r,K)$ whose cost if $z(r,K)$. 

Note that the number of sets in $S(l)$, with at least 2 elements, is $b-1$, since the tree representation of $l$ is a full binary tree. Moreover, we have to visit at most $n^2$ pairs of nodes for each set in $S(l)$, with cardinality of at least 2, where $n=|V|$. Consequently, if the number of terminals is fixed, the complexity of the proposed algorithm is $\mathcal{O}(n^2)$.

One can think of this algorithm as a simplified version of the algorithm proposed in \cite{dreyfus1971steiner}. The main difference is that the algorithm proposed \cite{dreyfus1971steiner} has to decide the way each set is split, while in this case, since we are in a particular laminar family, the split for every set is fixed.

\subsection{Solution Improvements}\label{Sec:Sol_Improvement}

For a given laminar family $l\in\mathcal{L}_b$, an optimal solution to $\mathcal{Z}_l$ may not correspond to a directed Steiner tree, but in the support of the solution we may have a directed Steiner tree, with a different structure, i.e., from a different laminar family in $\mathcal{L}_b$ (see Proposition 1 of \cite{siebert2019linear}). 

Figure \ref{Fig:Comparisson_Feas_Sol} shows an example where an optimal solution of a laminar family $l$ is not a directed Steiner tree. For the example, let $S(l)=\left\{\{k_1,k_2,k_3\},\{k_1,k_2\},\{k_1\},\{k_2\},\{k_3\}\right\}$.

\begin{figure}[H]
\begin{subfigure}{.5\textwidth}
  \centering
 \includegraphics[width=0.9\textwidth]{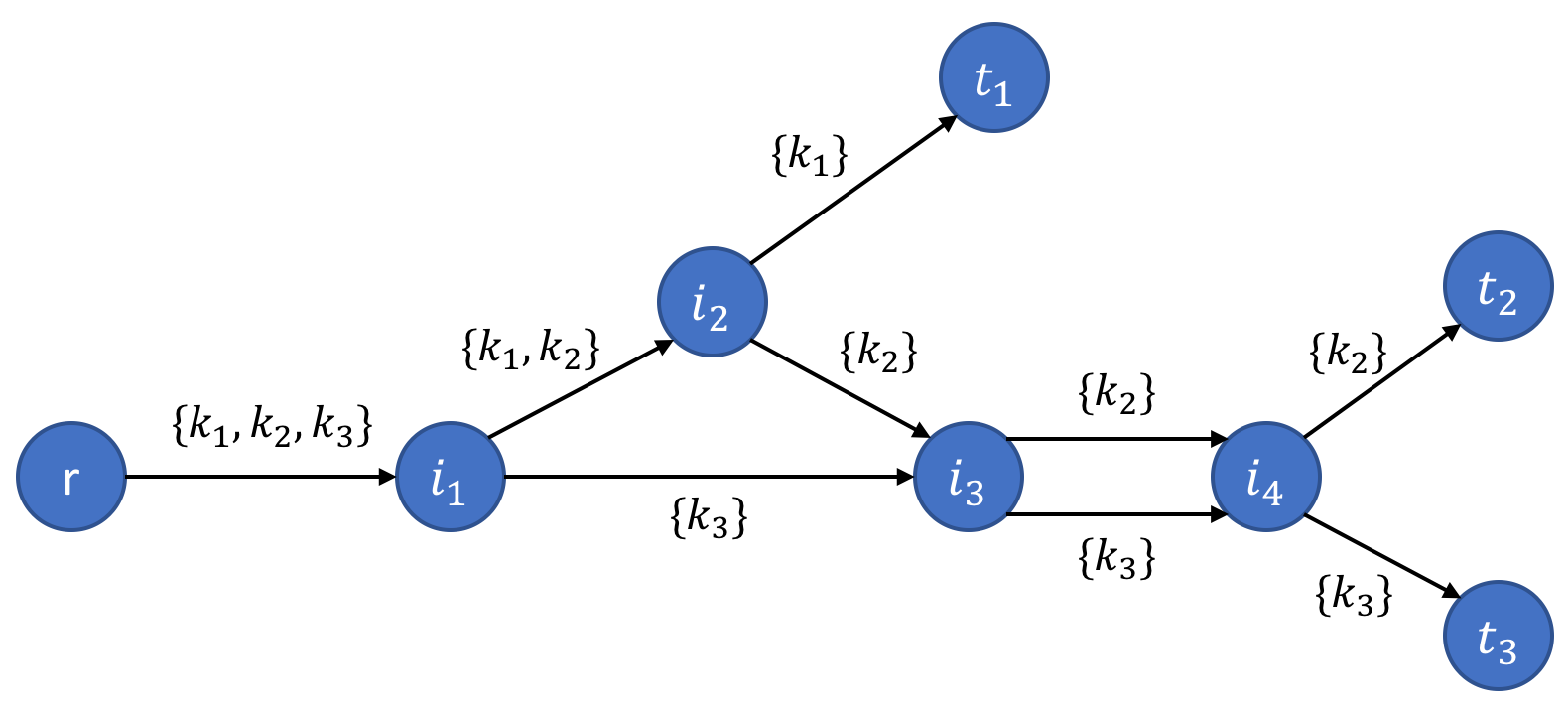}
  \caption{Optimal solution $x$ for $\mathcal{Z}_{l}$.}
  \label{fig:Feas1}
\end{subfigure}
\begin{subfigure}{.5\textwidth}
  \centering
 \includegraphics[width=0.9\textwidth]{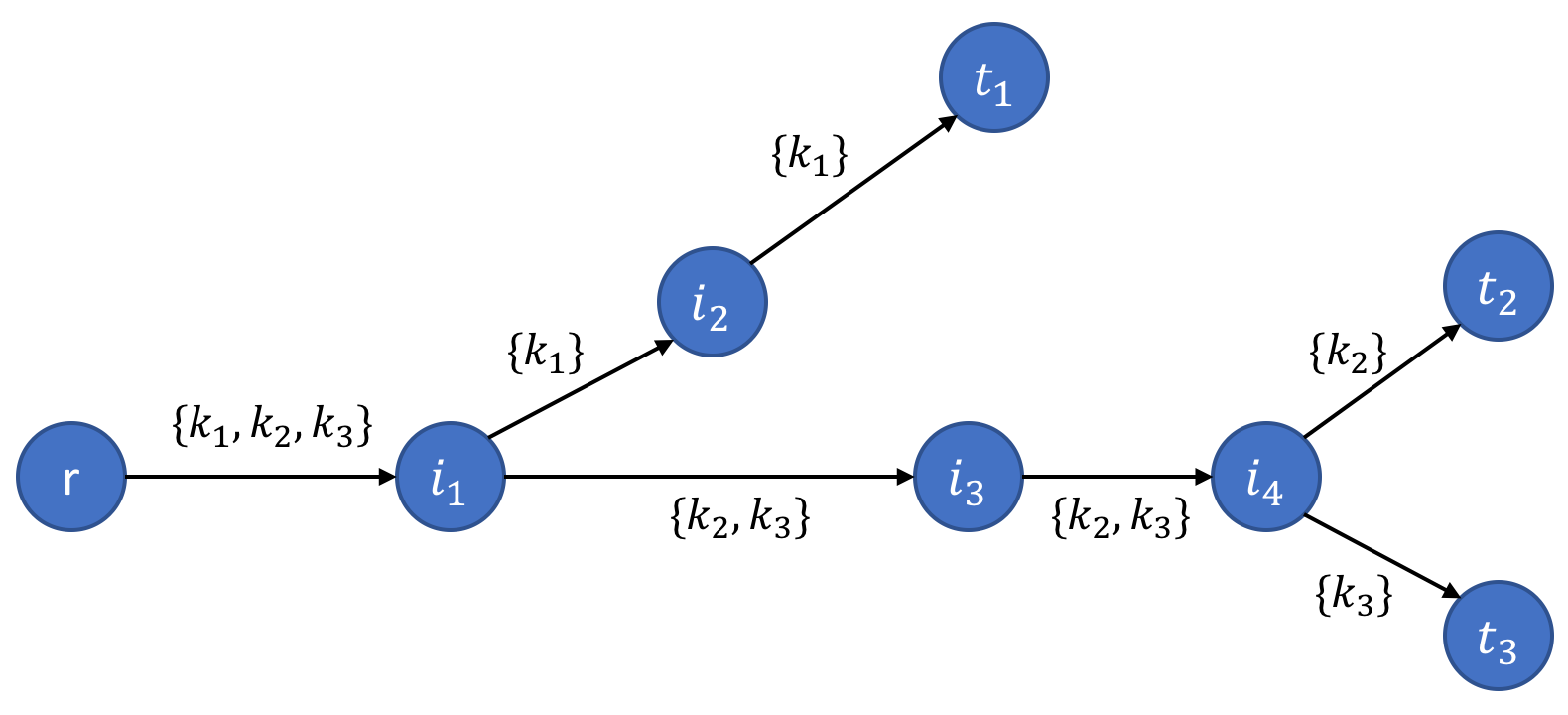}
  \caption{Steiner tree with different structure, contained in support of $x$.}
  \label{fig:Feas2}
\end{subfigure}
\caption{Optimal solution $x$ of $\mathcal{Z}_{l}$, which is not a directed Steiner tree, and a Steiner tree, with different structure contained in support of $x$.}
\label{Fig:Comparisson_Feas_Sol}
\end{figure}

We observe in Figure \ref{fig:Feas1}, an optimal solution $x$ of $\mathcal{Z}_{l}$, which is not a directed Steiner tree. Note that commodities $k_2$ and $k_3$ share arcs, but not using set $\{k_2,k_3\}$ because this set is not in $l$. In contrast, there is a Steiner tree, with $\hat{l}$ structure, contained in the support of $x$, where $S(\hat{l})=\left\{\{k_1,k_2,k_3\},\{k_2,k_3\},\{k_1\},\{k_2\},\{k_3\}\right\}$.

Since the cost vector is positive, then for every set $s\in S(l)$, there are no cycles in any optimal solution of $\mathcal{Z}_l$. Consequently, the only way the solution is not a directed Steiner tree is when, at least, two different sets reach the same node. Following the notation of \cite{siebert2019linear}, if $x=(f,\hat{y},\overline{y},w)$ is an optimal solution for $\mathcal{Z}_l$, then there exists $i\in V$, such that for two sets $s_1,s_2\in S(l)$, we have that $\sum_{a\in\delta^-(i)}f^{s_1}_a=1$ and $\sum_{a\in\delta^-(i)}f^{s_2}_a=1$. Whenever this happens, we have that the support of $f$ vector does not correspond to a directed tree rooted in $r$. 

By construction of the solution, there exists at least 1 path from $r$ to $t_k$ for all $k\in K$. Therefore, we can construct an $r$-arborescence using the arcs of the support of $f$, such that all the terminal nodes $t_k$ are reached from $r$, and that all leaves of such arborescence correspond to terminal nodes. Let $A(l)=\{a\in A:\sum_{s\in S(l)}f_a^s\geq 1\}$ be the support of an optimal solution $x$ to $\mathcal{Z}_l$, and let $T\subseteq A(l)$ be the set of arcs used by a minimum cost $r$-arborescence constructed in the subgraph defined by $A(l)$. If all leaves of $T$ are terminal nodes, then $T$ is a directed Steiner tree. On the other hand, if at least 1 leaf is a Steiner node, then we can prune such leaves and still have an $r$-arborescence with a path connecting $r$ to $t_k$ for all $k\in K$. If the tree after the first round of pruning still has leaves that are Steiner nodes, then we can prune such a tree again. After a finite number of rounds of pruning, all leaves of the $r$-arborescence will correspond to terminal nodes, which will correspond to a directed Steiner tree.

Once we have a rooted Steiner tree, we can identify its structure. We just have to see the way the paths from the root $r$ to every terminal $t_k$, share arcs. Since we are considering admissible laminar families, then the set of all commodities and all the singletons are always present. Consequently, by identifying the nodes where a split occurs, we can determine into which sets each set is partitioned, and therefore, identify the parent-child relationship between sets. 

It may happen that the laminar family of such a tree may not have a full binary tree representation, or in other words, we may have that a set has more than 2 children. In this case, there are several laminar families in $\mathcal{L}_b$ that have the constructed solution as a feasible solution. When this happens, we randomly select one of those laminar families as the laminar family of the constructed solution, using Algorithm \ref{Alg:Random_Laminar}. Lets call the selected laminar family $\hat{l}$. Finally, since the constructed solution may not be an optimal solution for $\hat{l}$, we solve the sub-problem $\mathcal{Z}_{\hat{l}}$ to get a new solution.

We describe the algorithm to construct a random laminar family from a laminar family, whose tree representation is not a full-binary tree. We refer to a laminar family $l$ as the collection of sets contained in the family, denoted by $S(l)$.

\begin{algorithm}[H]
\caption{Algorithm $randomLaminar(l)$}\label{Alg:Random_Laminar}
\begin{algorithmic}[1]
\Require Collection of sets $S(l)$.
\State Set $S'=\{s\in S(l):s\text{ has at least 3 children}\}$
\While{$S'\ne\emptyset$}
\State Select arbitrary $\hat{s}\in S'$
\State Let $\hat{S}=\{s\in S(l):s\text{ is child set of }\hat{s}\}$
\While{$|\hat{S}|\geq 3$}
\State Let $s_1$ and $s_2$ be randomly selected elements of $\hat{S}$
\State $s'\leftarrow s_1\cup s_2$
\State $S(l)\leftarrow S(l)\cup s'$
\EndWhile
\EndWhile
\Return $S(l)$.
\end{algorithmic}
\end{algorithm}

Algorithm \ref{Alg:Random_Laminar} creates a laminar family $\hat{l}$, with a full-binary tree representation, from an admissible laminar family $l$, whose tree representation is not a full-binary tree. For every set $\hat{s}$ in $S(l)$ that has at least 3 children, we take two random children $s_1$ and $s_2$ (step 6), and we create a new set from the union of $s_1$ and $s_2$, which is added to $S(l)$ (steps 7 and 8). Note that this reduces the number of child sets of $\hat{s}$ by 1, since $s_1\cup s_2$ is now a child of $\hat{s}$, while $s_1$ and $s_2$ are now children of $s_1\cup s_2$.

\section{Laminar family neighborhood characterization}\label{Sec:Neighborhood_Characterization}

Our final goal is to propose a local search heuristic algorithm to obtain good solutions to the directed Steiner tree problem. This local search algorithm starts from a given laminar family, and then moves to a promising neighbor laminar family. This process is performed several times. Each laminar family subproblem is solved in polynomial time using the algorithm presented in Section \ref{Sec:Algorithm_Section}. 

Consequently, we need to have a characterization of the neighborhood of each laminar family. Recall that, all the laminar families considered in the proposed approach have a unique corresponding full binary tree representation. Therefore, it suffices to define the neighborhood of a laminar family, as the neighborhood of its corresponding tree representation. For the rest of this section, we define $T$ to be a full binary tree, i.e., each node of $T$ has a degree 3 or 1.

In the literature, there are several ways proposed to characterize the neighborhood of a tree. The most commonly used are the Nearest Neighbor Interchange (NNI), the Subtree Pruning and Regrafting (SPR), and the Tree Bisection Reconnection (TBR) \cite{bryant2004splits,felsenstein2004inferring}. 

We decide to use the SPR process to construct the neighbors of a laminar family. In the studied literature, SPR was widely used over NNI and TBR. Furthermore, we consider the neighborhood size under SPR, which is $4(b-2)(b-3)$ \cite{felsenstein2004inferring}, to be adequate for our case. Finally, we consider that SPR is reasonably easy to be adapted to our setting.

The SPR process corresponds to selecting an edge $\{i,j\}$ of tree $T$. The edge $\{i,j\}$ is removed from $T$, dividing $T$ in two connected subtrees $T_i$ and $T_j$, containing $i$ and $j$ respectively. For simplicity of the argument, suppose that $T_i$ and $T_j$ have more than 3 nodes. Note that nodes $i$ and $j$ will have degree 2, after removing edge $\{i,j\}$. We leave $T_i$ as is, and we replace the two incident edges to $j$ in $T_j$, by an edge connecting the two neighbors of $j$ in $T_j$. Then, we select an edge of the updated $T_j$, and then we subdivide it creating a new node $k$. Finally, we create a new edge $\{i,k\}$ connecting node $i$ of subtree $T_i$, and node $k$ of subtree $T_j$. 

\begin{figure}[H]
\begin{subfigure}{.5\textwidth}
  \centering
 \includegraphics[width=0.7\textwidth]{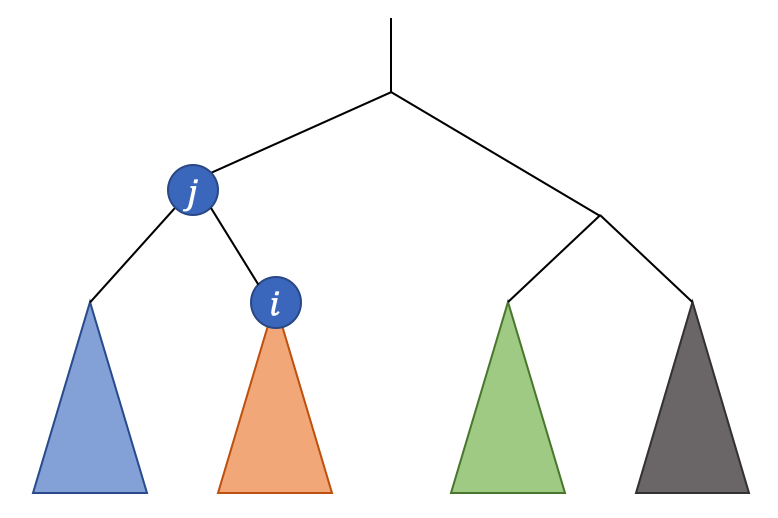}
  \caption{Tree $T$.}
  \label{Subfig:SPR_1}
\end{subfigure}
\begin{subfigure}{.5\textwidth}
  \centering
 \includegraphics[width=0.8\textwidth]{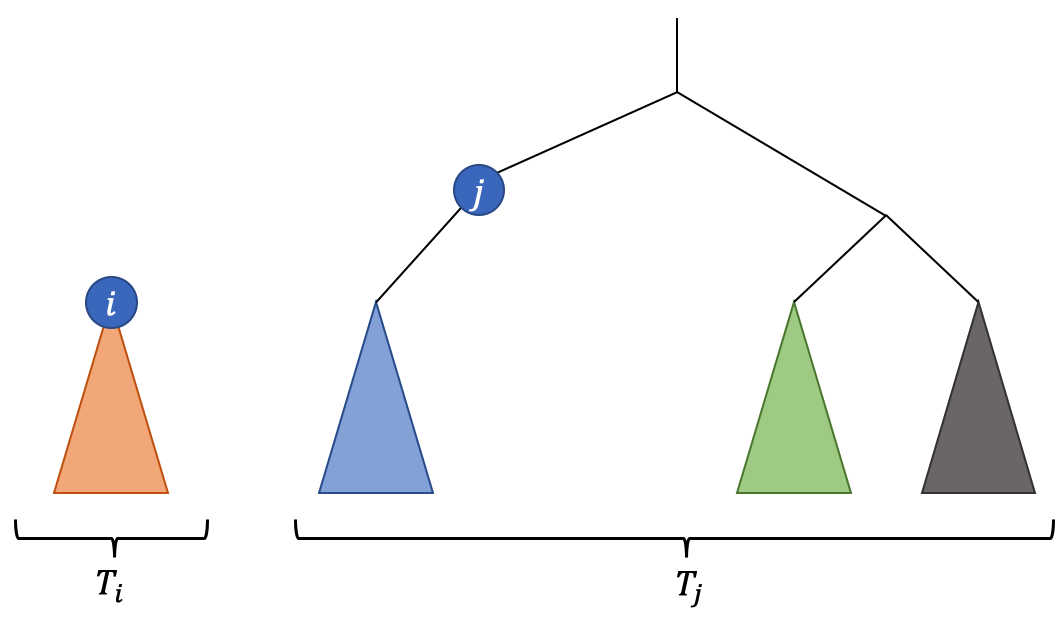}
  \caption{Subtrees $T_i$ and $T_j$, after removing edge $\{i,j\}$.}
  \label{Subfig:SPR_2}
\end{subfigure}\\
\begin{subfigure}{.5\textwidth}
  \centering
 \includegraphics[width=0.8\textwidth]{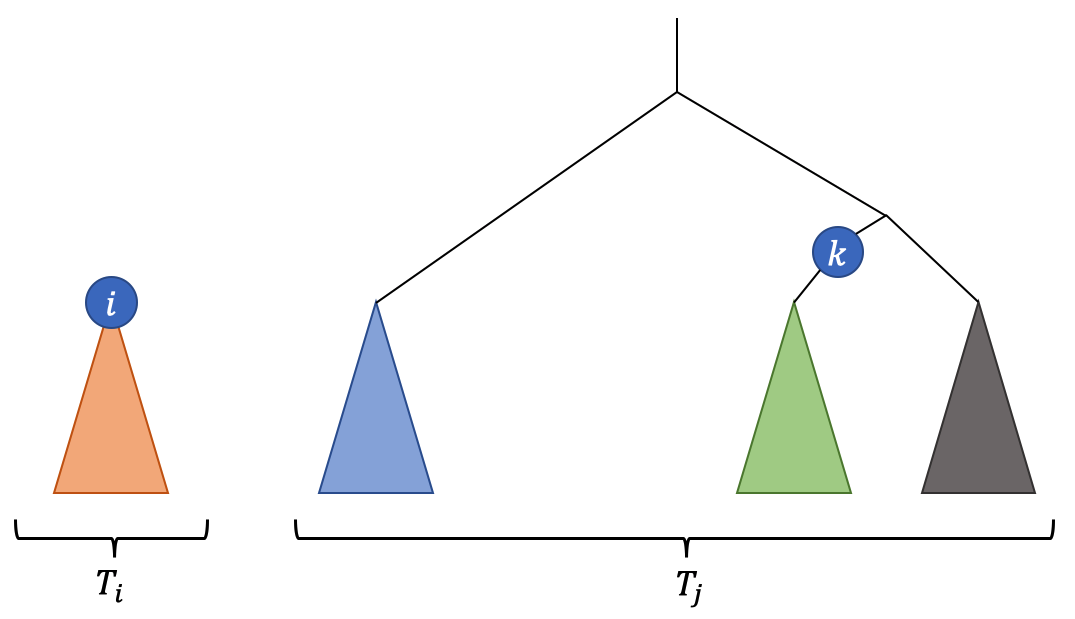}
  \caption{Original subtree $T_i$, and new subtree $T_j$.}
  \label{Subfig:SPR_3}
\end{subfigure}
\begin{subfigure}{.5\textwidth}
  \centering
 \includegraphics[width=0.7\textwidth]{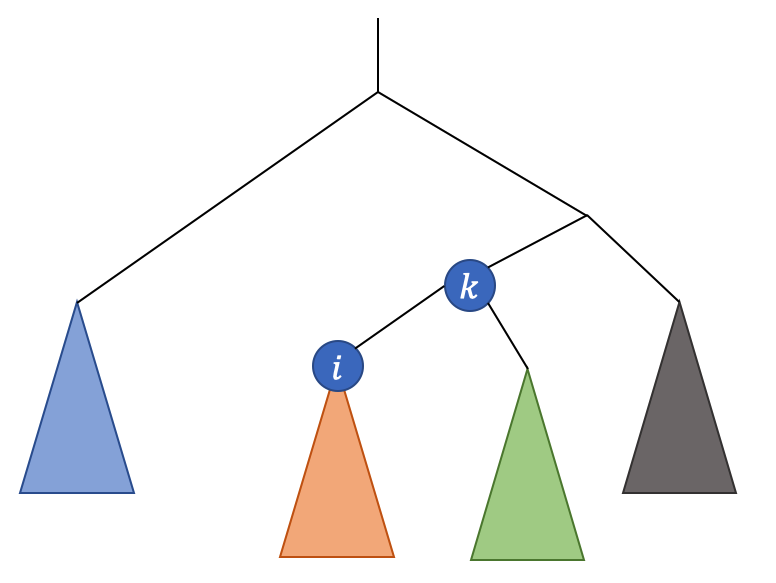}
  \caption{Neighbor of tree $T$ using SPR procedure.}
  \label{Subfig:SPR_4}
\end{subfigure}
\caption{A tree $T$ and a neighbor under the SPR procedure.}
\label{Fig:SPR}
\end{figure}

Figure \ref{Fig:SPR_In_Laminar} shows an example of the SPR process applied to our setting.

\begin{figure}[H]
\begin{subfigure}{.5\textwidth}
  \centering
 \includegraphics[width=0.8\textwidth]{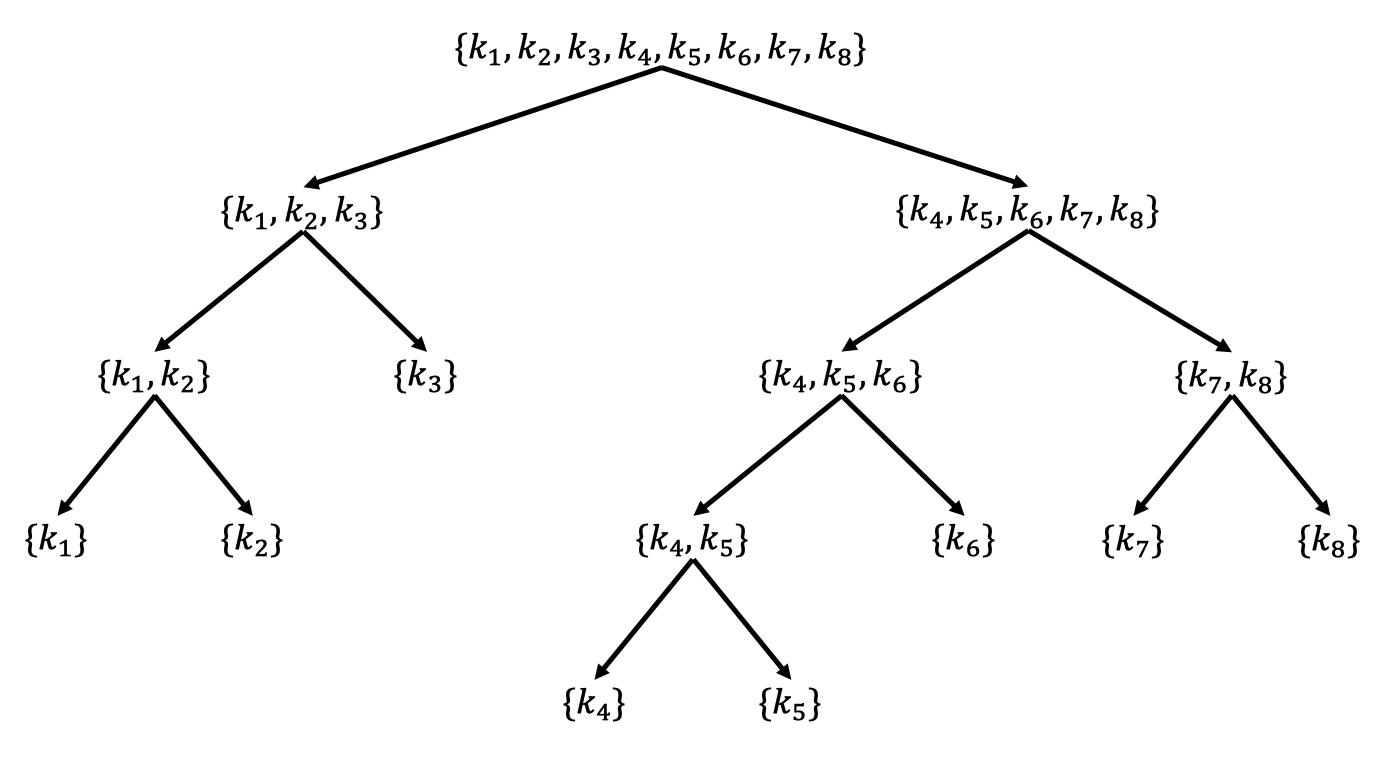}
  \caption{Original tree $T$.}
  \label{Subfig:SPR_Laminar_1}
\end{subfigure}
\begin{subfigure}{.5\textwidth}
  \centering
 \includegraphics[width=0.8\textwidth]{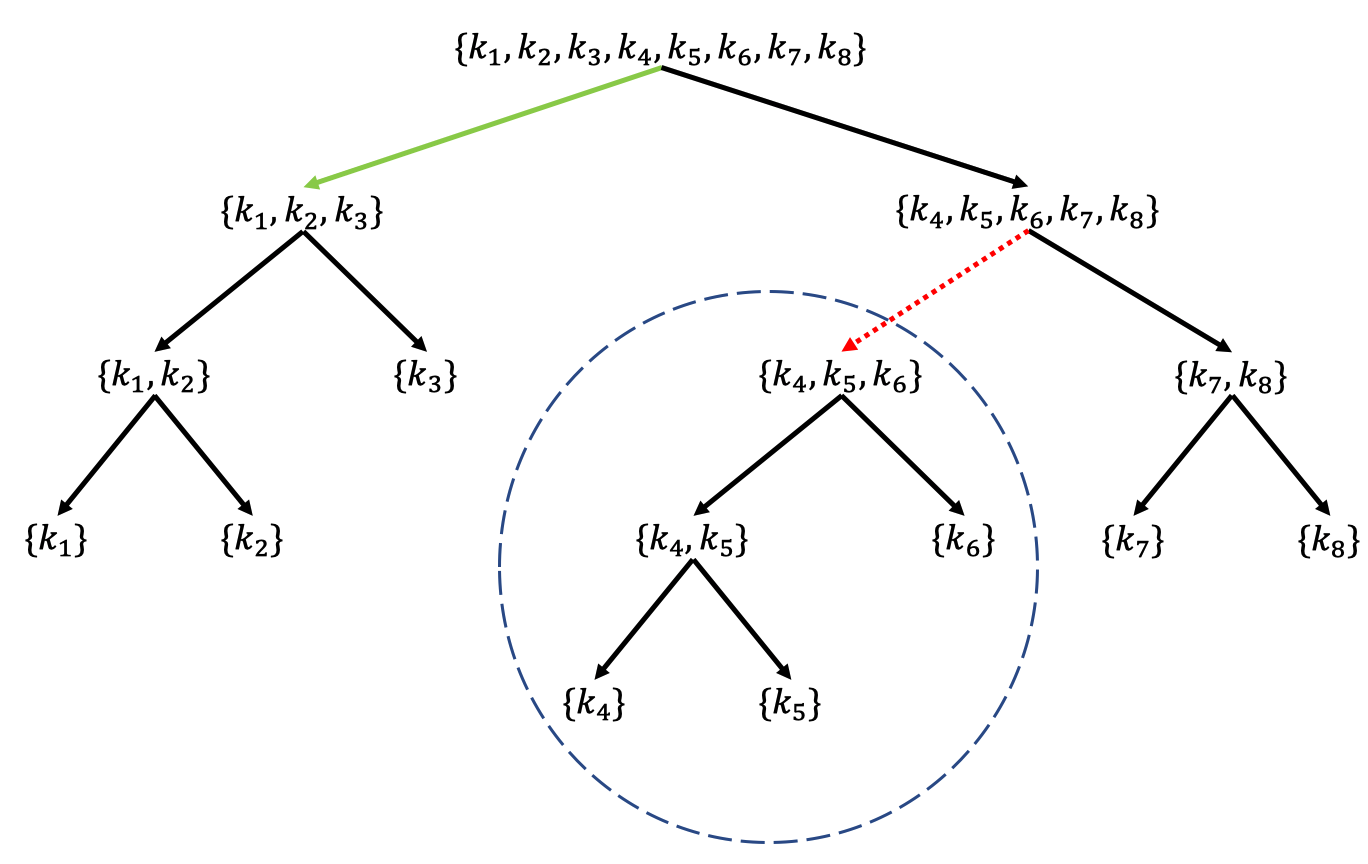}
  \caption{The subtree within dashed circle is pasted in green edge. The red dashed edge is removed.}
  \label{Subfig:SPR_Laminar_2}
\end{subfigure}\\
\begin{subfigure}{.5\textwidth}
  \centering
 \includegraphics[width=0.8\textwidth]{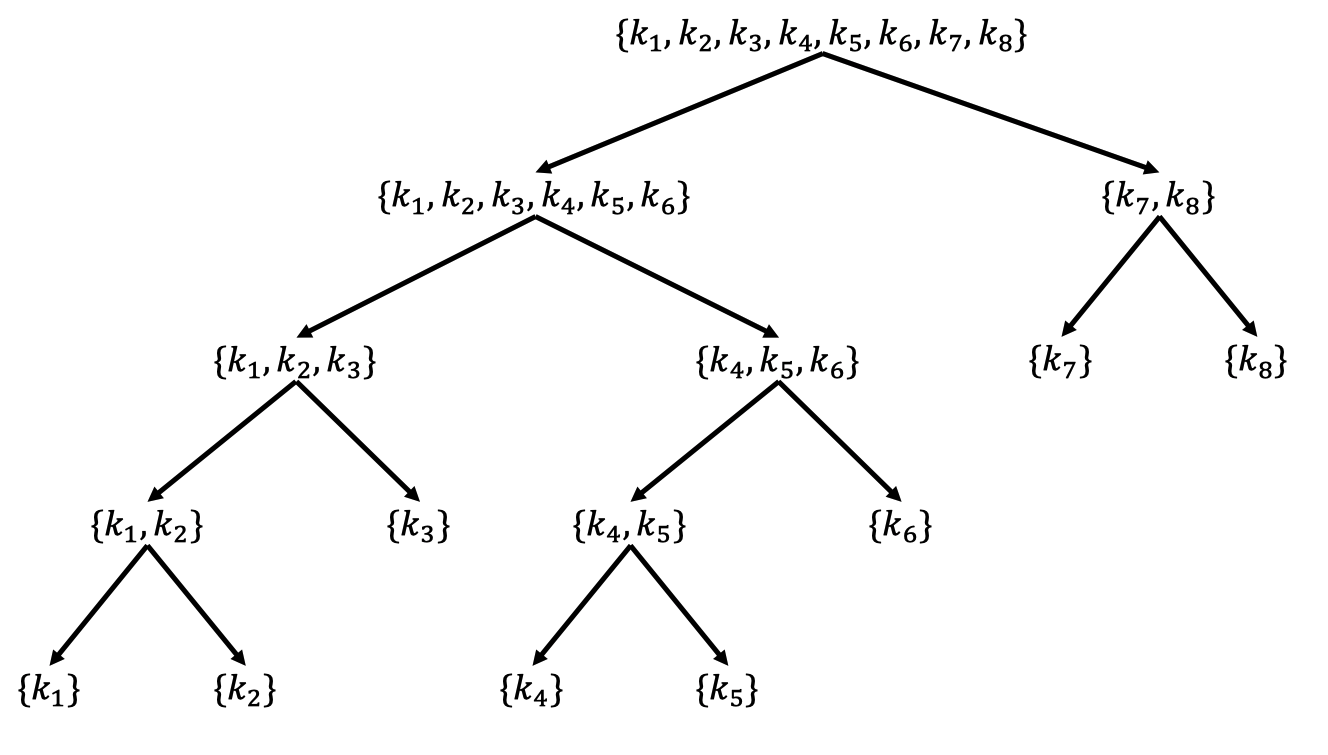}
  \caption{SPR neighbor of tree $T$.}
  \label{Subfig:SPR_Laminar_3}
\end{subfigure}
\caption{SPR neighbor of a laminar family.}
\label{Fig:SPR_In_Laminar}
\end{figure}

In Figure \ref{Subfig:SPR_Laminar_1}, we can see the original tree representation of a given laminar family. In Figure \ref{Subfig:SPR_Laminar_2}, we can see the subtree which is going to be removed and then regrafted, which is highlighted with the blue dashed circle. The red dashed edge is going to be removed, and the green edge is where the subtree is going to be regrafted. Figure \ref{Subfig:SPR_Laminar_3} shows the SPR neighbor, note that one set was removed $\{k_4,k_5,k_6,k_7,k_8\}$, and one was added $\{k_1,k_2,k_3,k_4,k_5,k_6\}$. This has to be done since the neighbor is the tree representation of a laminar family.

Note that we only need to make a few extra computations to have the optimal solution to the laminar family shown in Figure \ref{Subfig:SPR_Laminar_3}. Since we are using the algorithm presented in Section \ref{Sec:Algorithm_Section}, we only need to compute $z(i,\{k_1,\ldots,k_6\})$ for all $i\in V$, and recompute $z(r,K)$, since the child sets of $K=\{k_1,\ldots,k_8\}$ are different in the new laminar family. Everything else can be reused from the computations done to get an optimal solution of laminar family shown in Figure \ref{Subfig:SPR_Laminar_1}.

\section{Simulated annealing}\label{Sec:Simulated_Annealing}

Simulated annealing is a metaheuristic, which is widely used in local search frameworks \cite{kirkpatrick1983optimization,van1987simulated, aarts1988simulated,bertsimas1993simulated}. The main drawback when using local search algorithms, is that we often get stuck in local optima.

To address this in a minimization problem, simulated annealing also allows movements to neighbors with higher cost with a certain probability, which depends on how much worse the new the solution is, and the iteration of the algorithm. The high-level idea of simulated annealing is the following. We set an initial temperature $T_0$, and a starting initial solution. Then, at each iteration $j$, we decrease the temperature of the system having $T_j=f(T_0,j)$, where $f$ is a function depending on $T_0$ and $j$. Also, at each iteration we randomly choose a solution in the neighborhood of the current solution, and we compute the cost difference between the current solution and the neighbor, denoted by $\Delta_j$. If $\Delta_j<0$ then we move to the neighbor solution. If $\Delta_j\geq 0$ we move to the neighbor solution with probability $p(\Delta_j,T_j)$. Note that the probability is a function of the difference in cost of the solutions and the current temperature of the system.

There are three main decisions to make when using simulated annealing. How we set the initial temperature $T_0$, which function we use to reduce the temperature of the system, and what probability function we use to accept increasing cost neighbors. We use the probability function given by
\begin{align*}
p(\Delta_j,T_j)= \left[1+exp\left(\frac{\Delta_j}{T_j}\right)\right]^{-1} \in\left(0,\frac{1}{2}\right)
\end{align*}
Note that higher $\Delta_j$ leads to lower $p(\Delta_j,T_j)$, and smaller $T_j$ leads to lower acceptance probability. 

Since we want $\Delta_j$ and $T_j$ to be of the same order of magnitude, we use the cost of the first solution found as $T_0$. Finally, we use the temperature cooling function given by
\begin{align*}
T_j= f(T_0,j)=T_0(0.95)^j
\end{align*}

The simulated annealing framework is given in Algorithm \ref{Alg:Simulated_Annealing}.

\begin{algorithm}[H]
\caption{Simulated Annealing framework to solve directed Steiner tree problem}\label{Alg:Simulated_Annealing}
\begin{algorithmic}[1]
\State Set $N_{iter}$, $j=1$.
\State $l_{current}\leftarrow initial\_laminar()$, $x_{current}\leftarrow solve\_DP(l_{current})$, $c_{current}\leftarrow c(x_{current})$, $T_0\leftarrow c_{current}$
\State $l_{best}\leftarrow l_{current}$, $x_{best}\leftarrow x_{current}$, $c_{best}\leftarrow c_{current}$
\While {$j\leq N_{iter}$}
\State Set $T_j= f(T_0,j)$
\State $l_{new}\leftarrow SPR(l_{current})$
\State $x_{new}\leftarrow solve\_DP(l_{new})$
\State $c_{new}\leftarrow c(x_{new})$
\If {$c_{new}<c_{current}$}
\State $l_{current}\leftarrow l_{new}$
\State $x_{current}\leftarrow x_{new}$
\State $c_{current}\leftarrow c_{new}$
\If {$c_{new}<c_{best}$}
\State $l_{best}\leftarrow l_{new}$
\State $x_{best}\leftarrow x_{new}$
\State $c_{best}\leftarrow c_{new}$
\EndIf
\Else
\State $\Delta_j = c_{new}-c_{current}$
\State Sample $u \sim U(0,1)$
\If {$u\leq p(\Delta_j,T_j)$}
\State $l_{current}\leftarrow l_{new}$
\State $x_{current}\leftarrow x_{new}$
\State $c_{current}\leftarrow c_{new}$
\EndIf
\EndIf
\State $j\leftarrow j+1$
\EndWhile
\Return $x_{best}$
\end{algorithmic}
\end{algorithm}

The first 3 steps correspond to initialization of the algorithm. In the first step, we define the number of iterations $N_{iter}$, and set the iterations counter $j$ to 1. In the second step, we create an initial laminar family using Algorithm \ref{Alg:Single_Linkage_Clustering}, then we compute the best solution to the given laminar family, denoted by $x_{current}$, and the cost of the solution. We also define the initial temperature value. In the third step, we save the current solution as the best solution found so far.

Then, for each iteration $j$ we, first, compute the temperature given by function $f(T_0,j)$, and after that, we create a random SPR neighbor of the current laminar family, and we compute the best solution of the neighbor and its cost (steps 6 to 8). We check whether the solution of the neighbor is better than the current solution, and if it is, we update the current solution. Moreover, if the neighbor solution is better than the best solution found so far, we update the best solution (steps 9 to 16). If the solution of the neighbor is worse than the current solution, we compute the cost difference between the solutions and we sample a uniform $(0,1)$ random variable $u$. If the value of $u$ is less than or equal to the acceptance probability given by $p(\Delta_j,T_j)$, then we update the current solution (steps 17 to 23). Finally, we increase the iterations counter by 1.

The process to construct the initial structure to solve is the following. First, we create a complete graph $G'$ whose nodes are the terminal nodes. For every edge $e'=(t',t'')$, its cost is the length of the shortest path between $t'$ and $t''$ in the original graph $G$. Second, we construct a minimum spanning tree $T'$ in $G'$. Finally, we run Algorithm \ref{Alg:Single_Linkage_Clustering} to construct an initial laminar family using $T'$ as input.

\begin{algorithm}[H]
\caption{Algorithm $initial\_laminar()$, which creates an initial laminar family}\label{Alg:Single_Linkage_Clustering}
\begin{algorithmic}[1]
\Require Tree $T'$, we assume $T'$ is a set of edges ordered in increasing order of their lengths
\State Set $S(l)\leftarrow \{\{k_1\},\{k_2\},\ldots,\{k_b\}\}$
\For {$e'\in T'$}
\State Let $e'=(t',t'')$, and let $k'$ and $k''$ be the commodities of terminals $t'$ and $t''$, respectively.
\State Let $s'$ be largest set in $S(l)$ containing $k'$
\State Let $s''$ be largest set in $S(l)$ containing $k''$
\State $\hat{s}\leftarrow s'\cup s''$
\State $S(l)\leftarrow S(l)\cup\hat{s}$
\EndFor
\Return $S(l)$.
\end{algorithmic}
\end{algorithm}

Algorithm \ref{Alg:Single_Linkage_Clustering} constructs a laminar family $l$ based on the single linkage clustering algorithm, which is a widely used aglomerative clustering method \cite{xu2005survey}. It starts with the collection $S(l)$ of all the singletons. Then, it goes over the set of all edges of $T'$, which are assumed to be ordered in increasing order of their lengths. For each edge $e'=(t',t'')$, we take the commodities $k'$ and $k''$ of terminals $t'$ and $t''$, respectively. Then, we take sets $s'$ and $s''$, which are  the largest sets in $S(l)$ that contain $k'$ and $k''$, respectively. Finally, we create a set $\hat{s}$ to be the union of $s'$ and $s''$, and we add it to $S(l)$. After we visit all edges, we have that $S(l)$ is a collection of sets that forms an admissible laminar family.

\subsection{Simulated annealing with solution improvement}

As previously pointed out, sometimes an optimal solution to a laminar family subproblem may not be a directed Steiner tree, which is why it may be beneficial to improve the solution obtained at every iteration of the simulated annealing algorithm. We use the solution improvement routine introduced in Section \ref{Sec:Sol_Improvement} in the simulated annealing framework, which is presented below.

\begin{algorithm}[H]
\caption{Simulated Annealing for directed Steiner tree problem with solution improvement}\label{Alg:Simulated_Annealing_Tester}
\begin{algorithmic}[1]
\State Set $N_{iter}$, $j=1$.
\State $l_{current}\leftarrow random\_laminar()$, $x_{current}\leftarrow solve\_DP(l_{current})$, $c_{current}\leftarrow c(x_{current})$, $T_0\leftarrow c_{current}$
\State $l_{best}\leftarrow l_{current}$, $x_{best}\leftarrow x_{current}$, $c_{best}\leftarrow c_{current}$
\While {$j\leq N_{iter}$}
\State Set $T_j= f(T_0,j)$
\State $l_{new}\leftarrow SPR(l_{current})$
\State $x_{new}\leftarrow solve\_DP(l_{new})$
\State $c_{new}\leftarrow c(x_{new})$
\If {$A(l_{new})$ is not an $r$-arborescence}
\State $T\leftarrow MST(A(l_{new}))$
\State $l_{improved}\leftarrow LaminarFamily(T)$
\If {$l_{improved}$ is not full-binary tree}
\State $l_{improved}\leftarrow SampleFullBinaryTree(l_{improved})$
\EndIf
\State $l_{new}\leftarrow l_{improved}$
\State $x_{new}\leftarrow solve\_DP(l_{new})$
\State $c_{new}\leftarrow c(x_{new})$
\EndIf
\If {$c_{new}<c_{current}$}
\State $l_{current}\leftarrow l_{new}$
\State $x_{current}\leftarrow x_{new}$
\State $c_{current}\leftarrow c_{new}$
\If {$c_{new}<c_{best}$}
\State $l_{best}\leftarrow l_{new}$
\State $x_{best}\leftarrow x_{new}$
\State $c_{best}\leftarrow c_{new}$
\EndIf
\Else
\State $\Delta_j = c_{new}-c_{current}$
\State Sample $u \sim U(0,1)$
\If {$u\leq p(\Delta_j,T_j)$}
\State $l_{current}\leftarrow l_{new}$
\State $x_{current}\leftarrow x_{new}$
\State $c_{current}\leftarrow c_{new}$
\EndIf
\EndIf
\State $j\leftarrow j+1$
\EndWhile
\Return $x_{best}$
\end{algorithmic}
\end{algorithm}

Algorithm \ref{Alg:Simulated_Annealing_Tester} shows the pseudocode of simulated annealing with the solution improvement. The solution improvement, or tester, corresponds to lines 9 to 16. First, we check whether the solution of $l_{new}$ is an $r$-arborescence or not, in case it is not, then it can be improved. In step 10, we define $T$ to be the minimum cost $r$-arborescence whose leaves correspond to terminal nodes, as described above. In line 11, we define $l_{improved}$ to be the laminar family extracted from $T$. In line 12, we check whether $l_{improved}$ has a full binary tree representation or not. In case it is not full-binary, then we sample a full binary tree that contains $l_{improved}$. Finally, we update $l_{new}$, $x_{new}$ and $c_{new}$. 

\subsection{Rectilinear graphs}

The nodes of rectilinear graphs are placed in the $\R^2$ plane and the distance between nodes is given by the $\Vert\cdot\Vert_1$ distance. We can take advantage of this property since we can divide the plane into regions, each one containing a set of terminals, which will be used to create an initial laminar family. For instance, in Figure \ref{Fig:Rectilinear_Graph}, we can see that terminals $t_1$ and $t_7$ may not share arcs in an optimal solution, but it is very likely that $t_1$ shares arcs with $t_2$, and $t_7$ with $t_6$, since they are in similar regions of the plane.

The idea is to first, divide the set of terminal nodes in two sets, which will be interpreted as the first bipartition for the laminar family we want to construct. Then, each set of commodities is divided into another two sets, and so on, until all sets are singletons. In particular, the first partition corresponds to the partition given by clusterizing all the terminals into two clusters. Then each cluster is clusterized again into another two clusters, and so on, until each cluster has size 1. To illustrate this idea see Figure \ref{Fig:Rectilinear_Example}. In Figure \ref{Fig:Rectilinear_Graph}, we only show the root node and the terminal nodes of a rectilinear graph. Figure \ref{Fig:Rectilinear_Graph_Cluster_1} shows a potential partition based on the euclidean distance between terminals. In red dashed ovals, we find the first bipartition, inside each of the red zones, we find the second bipartition which is delimited by the green dashed oval, finally, inside each green zone with at least 2 terminals, we find a new partition denoted by the black dashed circles. Figure \ref{Fig:Rectilinear_Graph_Cluster_1_LaminarFamily} shows the tree representation of the laminar family associated with the partition shown in Figure \ref{Fig:Rectilinear_Graph_Cluster_1}.

\begin{figure}[H]
\centering
\begin{subfigure}[b]{.48\linewidth}
\includegraphics[width=\linewidth]{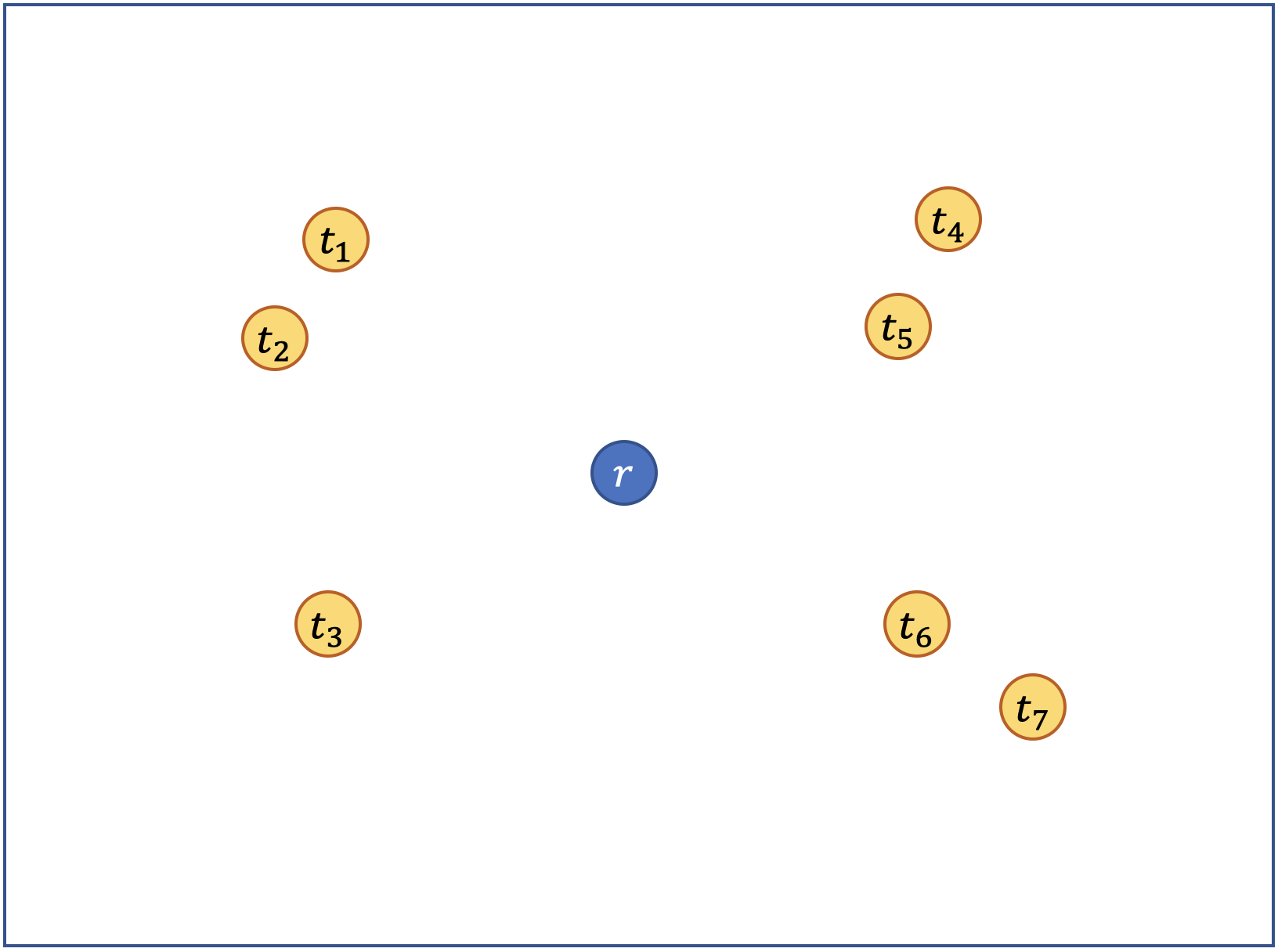}
\caption{Rectilinear graph, only showing root and terminal nodes}\label{Fig:Rectilinear_Graph}
\end{subfigure}

\begin{subfigure}[b]{.48\linewidth}
\includegraphics[width=\linewidth]{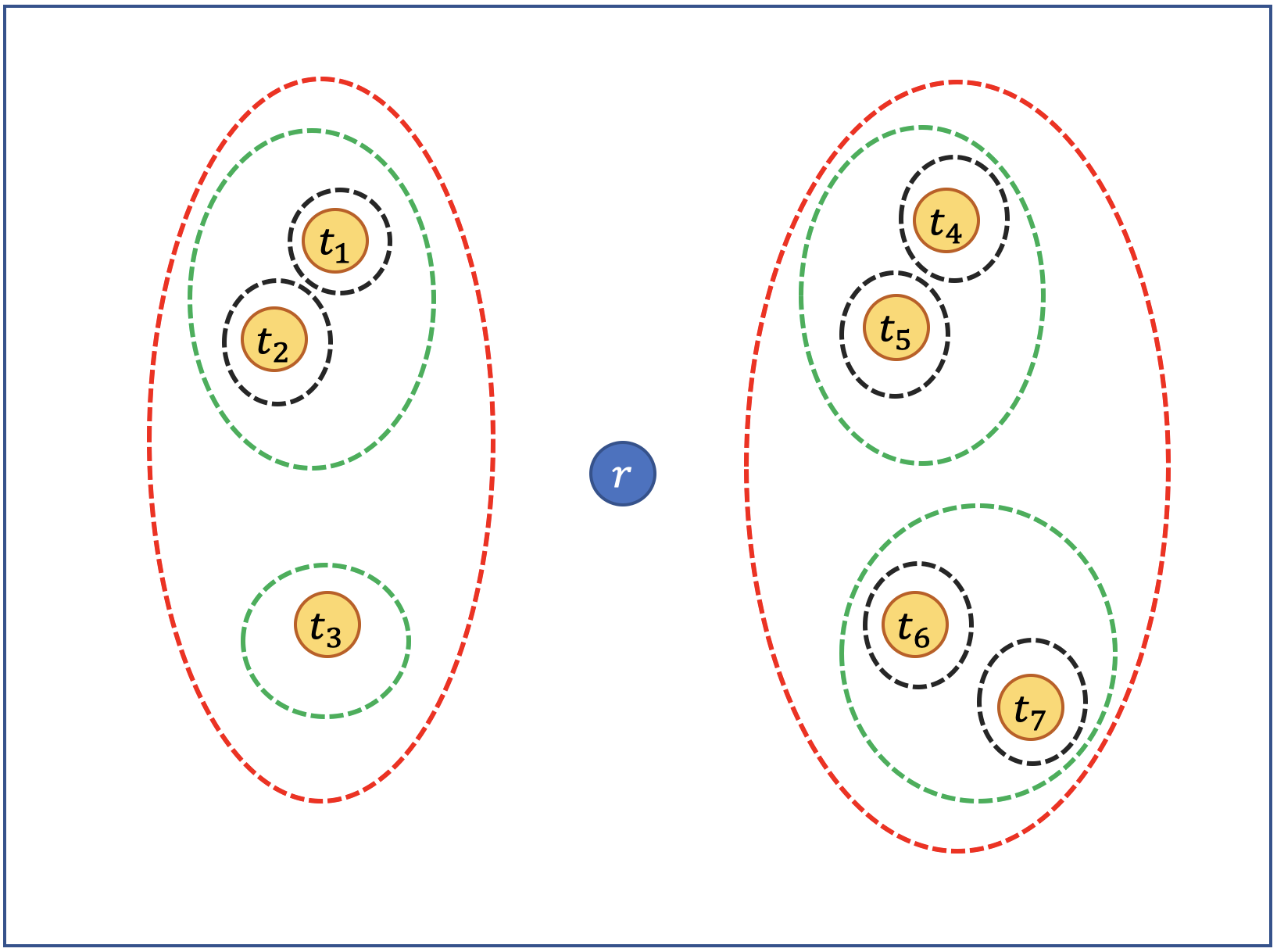}
\caption{Example of a clusterization of terminals in rectilinear graph}\label{Fig:Rectilinear_Graph_Cluster_1}
\end{subfigure}
\begin{subfigure}[b]{.48\linewidth}
\includegraphics[width=\linewidth]{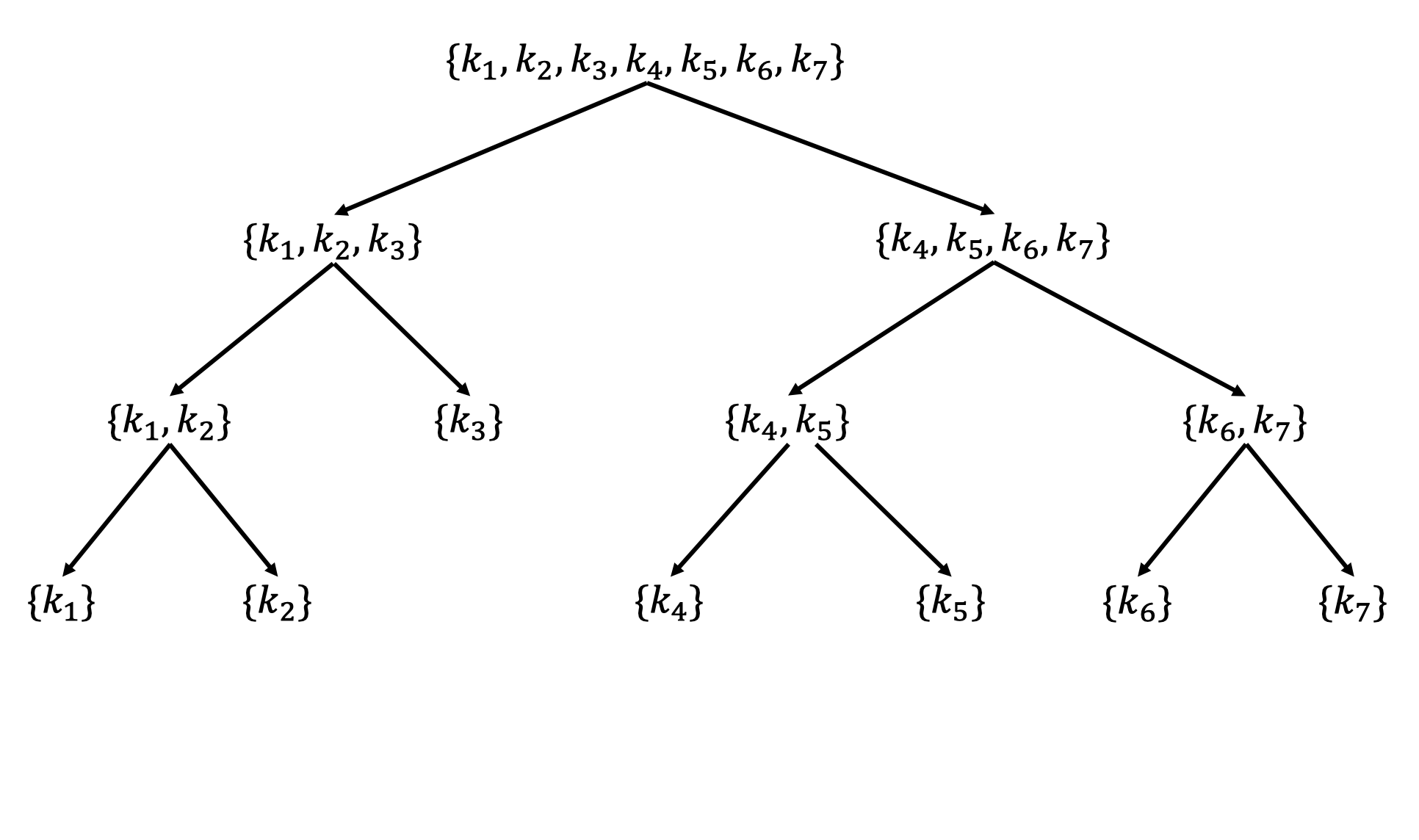}
\caption{Tree representation of laminar family of clusterization}\label{Fig:Rectilinear_Graph_Cluster_1_LaminarFamily}
\end{subfigure}
\caption{Example of terminals clusterization}
\label{Fig:Rectilinear_Example}
\end{figure}

There are many algorithms to cluster elements in the plane \cite{jain1999data}. We decided to use a $k$-means algorithm \cite{macqueen1967some} which is widely used in practice, because it is easy to implement, it runs very quickly, and it performs well in these types of clustering problems. Since the $k$-means output depends on the initial centroids, the quality of our constructed laminar family also depends on the initial centroids of each clusterization. This is why, we run the clustering-based algorithm several times to create laminar families, and then we select the laminar family with the best solution among all the candidates. 

\begin{algorithm}[H]
\caption{Algorithm $Part(s)$, which creates a bipartition of input set $s$, and all the created subsets}
\label{Alg:Clustering_Rect_Graph}
\begin{algorithmic}[1]
\Require Set $s$.
\State Define $\hat{S}=\{s\}$.
\If {$|s|=1$}
\Return $\hat{S}$
\ElsIf {$|s|=2$}
\State $\hat{S}\leftarrow \hat{S}\cup\left\{\{e_1\},\{e_2\}\right\}$, where $s=\{e_1,e_2\}$
\Else
\State $(s_1,s_2)=kMeans(s,2)$
\State $\hat{S}_1=Part(s_1)$,  $\hat{S}_2=Part(s_2)$
\State $\hat{S}\leftarrow \hat{S}\cup \hat{S}_1\cup \hat{S}_2$
\EndIf
\Return $\hat{S}$
\end{algorithmic}
\end{algorithm}

Algorithm \ref{Alg:Clustering_Rect_Graph}, named $Part(s)$, determines the way a new laminar family is created based on the clustering process previously described. This algorithm takes as input a given set $s$. If the set has only one element, then a collection of sets $\hat{S}$, containing the singleton set $s$, is returned as shown in line 4. If set $s$ has exactly 2 elements, then the algorithm returns a collection of sets $\hat{S}$, containing set $s$ and the two singletons, as shown in line 6. If set $s$ has 3 or more elements, then $s$ is split into two subsets $s_1$ and $s_2$, which is the result of running the $k$-means algorithm for $k=2$, i.e., two clusters. Then, the algorithm recurses getting two collections of sets $\hat{S}_1$ and $\hat{S}_2$, which is the result of applying algorithm $Part(\cdot)$ to $s_1$ and $s_2$ respectively. The returned collection of sets $\hat{S}$, contains all sets in $\hat{S}_1$, all sets in $\hat{S}_2$, and the set $s$. Consequently, to get the desired laminar family, we have to run algorithm $Part(s)$, for $s=K$.

When the instance we want to solve is undirected, then we need to choose the root node among the terminal nodes. For all $i\in R$, let $L(i)$ and $R(i)$ be the number of terminal nodes to the left of $i$, and to the right of $i$, respectively. And, let $U(i)$ and $B(i)$ be the number of terminals that are above $i$, and below $i$ in the plane, respectively. Finally, let $\Delta_x(i)=|R(i)-L(i)|$ and $\Delta_y(i)=|U(i)-B(i)|$, then we choose $r$ as follows.
\begin{align*}
r=\argmin_{i\in R}\{\Delta_x(i)+\Delta_y(i)\}
\end{align*}
We are basically choosing $r$ to be the most centrical terminal node. We pick $r$ in this fashion, since we can have a better guess of the structure of an optimal directed Steiner tree. For instance, in the example shown in Figure \ref{Fig:Rectilinear_Graph}, it is very likely that the paths from $r$ of the terminals that are above and to the left of $r$, i.e., terminals $t_1$ and $t_2$, are not going to share many arcs with the paths from $r$ to terminals that are below and to the right of $r$, i.e., terminals $t_6$ and $t_7$.

For rectilinear graphs, we use algorithm  \ref{Alg:Clustering_Rect_Graph} to create the initial laminar family in the simulated annealing framework. As described in Algorithms \ref{Alg:Simulated_Annealing} and \ref{Alg:Simulated_Annealing_Tester}, the initial laminar family is created at random. In rectilinear graphs, we run Algorithm \ref{Alg:Clustering_Rect_Graph} several times (the number of clusterizations is a parameter of the algorithm), and then we select a laminar family from the set of laminar families with the lowest optimal cost. The rest of the algorithm is the same as algorithm \ref{Alg:Simulated_Annealing}, or algorithm \ref{Alg:Simulated_Annealing_Tester} if we use the solution improvement. Moreover, if the instance is undirected, we choose $r$ as the most centrical terminal, as previously described.

\subsection{Worst case analysis}

Let $OPT$ be the value of an optimal solution for the directed Steiner tree problem instance, let $OPT_{SA}$ be the value of the solution returned by the simulated annealing framework, and for $l\in\mathcal{L}_b$, let $OPT(l)$ be the value of an optimal solution to $\mathcal{Z}_l$. 

\begin{lemma}\label{Lemma:worst_case}
For any $l\in\mathcal{L}_b$ we have $OPT\leq |R|\times OPT(l)$
\end{lemma}

\begin{proof}
Let $l$ be an arbitrary laminar family of $\mathcal{L}_b$. Note that we can always construct the following feasible solution. For all $s\in S(l)$ with $|s|\geq 2$, we fix $f_a^s=0$ for all $a\in A$, and for all $k\in K$, which correspond to $s\in S(l)$ with $|s|=1$, we take a shortest path from $r$ to $t_k$. It is known that such solution is at most $|R|$ times the value of an optimal solution of the problem \cite{takahashi1990approximate}. Consequently, any optimal solution to $\mathcal{Z}_l$ is at most $|R|$ times the value of an optimal solution to the problem.
\end{proof}

\begin{proposition}
For every instance of the directed Steiner tree problem, we have $OPT\leq |R|\times OPT_{SA}$
\end{proposition}

\begin{proof}
Since at every iteration we solve to optimality $\mathcal{Z}_l$ for some $l\in\mathcal{L}_b$, then using Lemma \ref{Lemma:worst_case}, we have that at every iteration, the best solution found cannot be more than $|R|$ times the value of an optimal solution to the problem, and the statement holds.
\end{proof}

Let $T(l)$ be the tree representation of laminar family $l\in\mathcal{L}_b$, and let $l^*\in\mathcal{L}_b$ be the laminar family of an optimal solution to the problem . We define $d_{SPR}(T(l_1),T(l_2))$ as the minimum number of SPR moves to transition from tree $T(l_1)$ to tree $T(l_2)$. It has been proven that computing $d_{SPR}(T(l_1),T(l_2))$ for any pair of $l_1,l_2$ is NP-Hard \cite{bordewich2005computational}. Nevertheless, it was proven in \cite{song2003combinatorics}, that for any two $l_1,l_2\in\mathcal{L}_b$, we have that $d_{SPR}(T(l_1),T(l_2))\leq |R|-2$. Consequently, if we use at least $|R|-2$ iterations in the simulated annealing framework, no matter the initial laminar family chosen, the probability of solving $\mathcal{Z}_{l^*}$ in a given iteration, is strictly positive.

\section{Computation experiments}\label{Sec:Results}

In this section we present our computational results. We compare our proposed simulated annealing framework with all the algorithms studied in \cite{watel2016practical}.

In \cite{watel2016practical}, 6 algorithms are compared. The first algorithm, denoted by $\text{ShP}_1$, takes a shortest path from $r$ to each terminal, and then returns the union of such shortest paths as a solution. The second algorithm, denoted by $\text{ShP}_2$, takes a shortest path from $r$ to its closest terminal, sets the cost of all used arcs to 0, and then proceeds in the same fashion with the rest of the terminals, until all terminals are reached. The third algorithm, denoted by DuAs, corresponds to using the dual ascent algorithm presented in \cite{wong1984dual}. If the solution is fractional, then it takes a minimum cost spanning tree within the support of the fractional solution. If some of the leave nodes are non-terminal nodes, then the tree is pruned until all leaves are terminal nodes. The fourth algorithm, denoted by Roos, corresponds to the algorithm presented in \cite{charikar1999approximation}, which has an approximation ratio $\mathcal{O}\left(|R|^{\frac{1}{t}}\right)$ using $t=2$, and is implemented using Roos modified algorithm which is described in \cite{ming2006fasterdsp}. The fifth algorithm, denoted by FLAC, is one of the algorithms introduced by the authors in \cite{watel2016practical}. This algorithm takes each arc as a pipe with capacity, in liters, equal to its cost. Then the algorithm tries to send water to the terminals, at a rate of 1 liter of water per second. Initially, only the arcs incoming to terminals are going to be considered. When an arc is up to its capacity, it is said to be saturated. Once an arc is saturated, then we start looking at the arcs incoming to the tail node of the saturated arcs too. This process is done until we reach the root node. Then, within the support of the saturated arcs, we have a directed Steiner tree. The sixth algorithm, denoted by $\text{FLAC}^{\rhd}$, is the FLAC algorithm applied to the shortest path instance of the problem, i.e., to a complete directed graph where the cost from $u$ to $v$ is given by the shortest path, in the original graph, between $u$ and $v$. Finally, we define the Best Benchmark (BB) to be the algorithm that, for each instance, takes the best result among the previous 6 algorithms. The solutions provided by our algorithms will be compared with the solutions provided by BB.

On the other hand, we present 3 algorithms, SA, SA-Test and SA-Rect algorithm. The SA algorithm corresponds to Algorithm \ref{Alg:Simulated_Annealing}, SA-Test corresponds to Algorithm \ref{Alg:Simulated_Annealing_Tester}, and SA-Rect corresponds to Algorithm \ref{Alg:Simulated_Annealing_Tester} but using Algorithm \ref{Alg:Clustering_Rect_Graph} to construct the initial laminar family. Since all of the proposed algorithms have a random component, we run 10 replications of each one, and the result of each algorithm corresponds to the solution with the lowest cost among the 10 replications. Furthermore, we run each algorithm with 1,000 and with 5,000 iterations, since the number of iterations is a parameter of the proposed algorithms. In the SA-Rect case, we run 50 replications of Algorithm \ref{Alg:Clustering_Rect_Graph}, and we select the solution with lowest cost as the initial laminar family of the simulated annealing algorithm.

We run experiments using the same instances studied in \cite{watel2016practical}, which correspond to directed graphs constructed based on undirected instances from the SteinLib library \cite{koch2001steinlib}. We only consider instances with at most 160 terminal nodes, and less than 3,500 nodes, since the proposed approach requires the computation of the shortest path between every pair of nodes. In total, we studied just over 800 instances, whose details can be found in Table \ref{app:Table:Studied_Instances} of Appendix \ref{app:instances}. 

The proposed algorithms were implemented in Java 8. All the experiments were run in an AWS {\it c5d.2xlarge} machine, with am {\it Intel Xeon Platinum 8000-series} processor (3.0 GHz) of 8 cores and 16 GB of memory.

\subsection{Non-rectilinear graphs}\label{Sec:Results_Non_Rectilinear_Graphs}

In this section, we compare the solution quality of the regular simulated annealing algorithm (see Algorithm \ref{Alg:Simulated_Annealing}), the simulated annealing with solution improvement (see Algorithm \ref{Alg:Simulated_Annealing_Tester}), and the best benchmark algorithm. We only consider instances which are not rectilinear graphs, since there is a specific algorithm for such cases, which are analyzed in section \ref{Sec:Results_Rectilinear_Graphs}. We use the performance profile approach to compare the solution quality delivered by the studied algorithms \cite{dolan2002benchmarking}. Figure \ref{Fig:SA_vs_SA_Test} shows the cumulative distribution of instances versus the error of the obtained solution, this graph should be read as follows. If, for instance, we have a point $(1.5,0.85)$ it means that 85\% of the instances solved have a gap smaller or equal to 1.5\%.

\begin{figure}[H]
\begin{center}
\includegraphics*[width=1.0\textwidth]{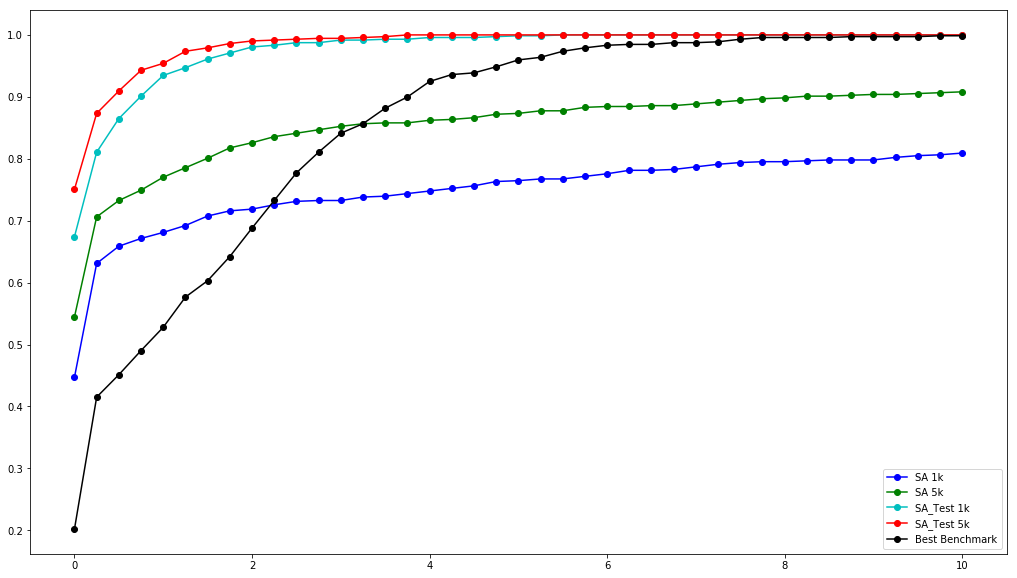}
\caption{Performance profile for SA and SA-Test, for 1,000 and 5,000 iterations}
\label{Fig:SA_vs_SA_Test}
\end{center}
\end{figure}

It is clear from Figure \ref{Fig:SA_vs_SA_Test} that SA-Test outperforms SA, and BB. First of all, we expected that, for the same algorithm, using more iterations will lead to better results. This happens with both, the SA and the SA-Test algorithm. The difference is more pronounced in the SA case; while in the SA-Test, the difference still exists, specially for smaller gaps, but after 2\%, the two curves are almost identical. Also, note that SA-Test algorithm with just 1,000 iterations, outperforms by far the BB algorithm, and the SA algorithm with 5,000 iterations. Another interesting observation is the difference in the proportion of instances solved to optimality. The BB algorithm solves around 20\% of the instances to optimality, being the worst of all the algorithms studied in this topic. Furthermore, while the SA algorithm cannot solve more than 55\% of the instances to optimality, we have that SA-Test can solve over 67\% of the instances to optimality with 1,000 iterations, and 75\% of the instances with 5,000.

\begin{table}[H]
\begin{center}
\begin{tabular}{c |c c c}
\hline
	Algorithm & Worse than SA 5k & Equal to SA 5k & Better than SA 5k \\  \hline
	SA-Test 1k & 4.7\% & 55.4\% & 39.8\% \\
	SA-Test 5k & 1.8\% & 55.8\% & 42.3\% \\
	 \hline
\end{tabular}
\end{center}
\caption{Proportion of instances where SA-Test is strictly worse, equal, or strictly better than SA with 5,000 iterations}
\label{Table:SA_vs_SA_Test}
\end{table}
Table \ref{Table:SA_vs_SA_Test} shows the proportion of instances where SA-Test algorithm does strictly worse, equal, or strictly better than SA with 5,000 iterations. Even when we use 1,000 iterations, SA-Test only delivers worse results than SA in fewer than 5\% of the cases studied; this number reduces to 1.8\% when we use 5,000 iterations in SA-Test.
\begin{table}[H]
\begin{center}
\begin{tabular}{c |c c c}
\hline
	Algorithm & Worse than BB & Equal to BB & Better than BB \\  \hline
	SA-Test 1k & 3.6\% & 21.0\% & 75.3\% \\
	SA-Test 5k & 2.2\% & 20.6\% & 77.2\% \\
	 \hline
\end{tabular}
\end{center}
\caption{Proportion of instances where SA-Test is strictly worse, equal, or strictly better than best benchmark}
\label{Table:BB_vs_SA_Test}
\end{table}
Table \ref{Table:BB_vs_SA_Test} show the proportion of instances where SA-Test algorithm does strictly worse, equal, or strictly better than BB. When we use 1,000 iterations, SA-Test only deliver worse results than BB in 3.6\% of the cases studied, this number reduces to 2.2\% when we use 5,000 iterations in SA-Test.

We conclude that SA-Test outperforms, in solution quality, the BB and SA algorithms. Although SA-Test with 5,000 iterations gives better performance than SA-Test with 1,000 iterations, the results are not considerably better.

\subsection{Rectilinear graphs}\label{Sec:Results_Rectilinear_Graphs}

In this section we consider rectilinear graphs. We compare the solution quality of the regular simulated annealing algorithm (see Algorithm \ref{Alg:Simulated_Annealing}), the simulated annealing for rectilinear graphs (see Algorithms \ref{Alg:Simulated_Annealing_Tester} and \ref{Alg:Clustering_Rect_Graph}), and the best benchmark algorithm. We compare the three algorithms in the same way we compared SA, SA-Test, and BB in section \ref{Sec:Results_Non_Rectilinear_Graphs}.

\begin{figure}[H]
\begin{center}
\includegraphics*[width=1.0\textwidth]{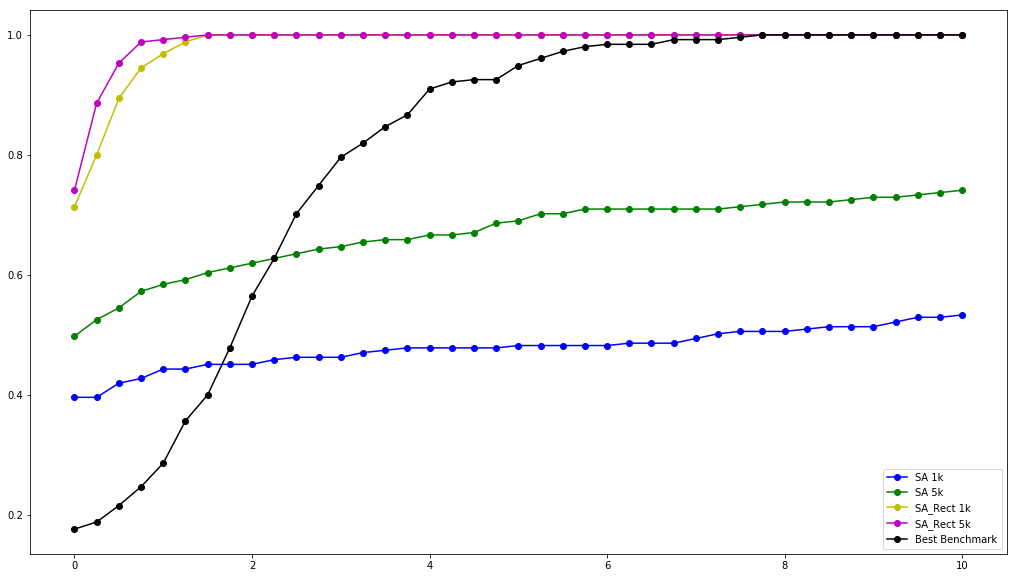}
\caption{Performance profile for SA and SA-Rect, for 1,000 and 5,000 iterations}
\label{Fig:SA_vs_SA_Rect}
\end{center}
\end{figure}

It is clear from Figure \ref{Fig:SA_vs_SA_Rect} that SA-Rect outperforms SA, and BB. The results are even more pronounced than the non-rectilinear graphs, since SA-Rect always delivers solutions with at most 1.75\% gap. With respect to the SA and SA-Rect algorithms, again we see that by solving the problem with more iterations, the quality of the solutions improve. And again, the difference is more pronounced in the SA case. Indeed, for the SA-Rect the performance profiles are very similar, having a slightly better performance with 5,000 iterations.

In this case, we also have that SA-Rect algorithm with just 1,000 iterations, outperforms by far the BB algorithm, and the SA algorithm with 5,000 iterations. BB solves less than 18\% of the instances to optimality, while SA solves almost 40\% of instances to optimality with 1,000 iterations, and almost 50\% with 5,000 iterations. In contrast, SA-Rect solves over 71\% of instances to optimality with 1,000 iterations, and almost 75\% with 5,000 iterations.

\begin{table}[H]
\begin{center}
\begin{tabular}{c |c c c}
\hline
	Algorithm & Worse than SA 5k & Equal to SA 5k & Better than SA 5k \\  \hline
	SA-Rect 1k & 1.2\% & 49.8\% & 49.0\% \\
	SA-Rect 5k & 1.2\% & 50.2\% & 48.6\% \\
	 \hline
\end{tabular}
\end{center}
\caption{Proportion of instances where SA-Rect is strictly worse, equal, or strictly better than SA with 5,000 iterations}
\label{Table:SA_vs_SA_Rect}
\end{table}
Table \ref{Table:SA_vs_SA_Rect} shows the proportion of instances where SA-Rect algorithm does strictly worse, equal, or strictly better than SA with 5,000 iterations. We can see that when the number of iterations is 1,000, SA-Rect only delivers worse results than SA in 1.2\% of the cases studied, the same amount when we use 5,000 iterations in SA-Rect.
\begin{table}[H]
\begin{center}
\begin{tabular}{c |c c c}
\hline
	Algorithm & Worse than BB & Equal to BB & Better than BB \\  \hline
	SA-Rect 1k & 0.4\% & 17.3\% & 82.3\% \\
	SA-Rect 5k & 0.0\% & 17.6\% & 82.4\% \\
	 \hline
\end{tabular}
\end{center}
\caption{Proportion of instances where SA-Rect is strictly worse, equal, or strictly better than best benchmark}
\label{Table:BB_vs_SA_Rect}
\end{table}

Table \ref{Table:BB_vs_SA_Rect} shows the proportion of instances where SA-Rect algorithm does strictly worse, equal, or strictly better than BB. We can see that when we use 1,000 iterations, SA-Rect only delivers worse results than BB in 0.4\% of the cases studied, which in this case corresponds to only one instance. When we use 5,000 iterations, SA-Rect performance is at least as well as BB in all the studied instances.

We conclude, that SA-Rect outperforms, in solution quality, the BB and SA algorithms. Although the conclusions in the rectilinear case are similar to the non-rectilinear case, the solution quality obtained by the improved version of SA in rectilinear graphs is better than the non-rectilinear case. All the instances in the rectilinear case present a gap smaller or equal to 1.75\%, while in the non-rectilinear case it is 3.5\% when we run SA-Test with 5,000 iterations, and 5\% with 1,000 iterations. Moreover, the proportion of instances where SA-Rect performs worse than, either SA or BB, is lower than SA-Test.

\subsection{Execution times}

At each iteration of simulated annealing, we compute the new solution based on the solution of the previous iteration. Therefore, simulated annealing, by nature, is a sequential algorithm. There are some researchers that study parallel versions of the algorithm \cite{greening1990parallel,ram1996parallel,czech2002parallel}, but we just focused on the original version. Consequently, the running time of the algorithm will depend on the number of iterations.

Figure \ref{Fig:SA_Time_Profile} shows the histogram of the average execution time per iteration for simulated annealing with, and without the solution improvement routine.

\begin{figure}[H]
\begin{center}
\includegraphics*[width=1.0\textwidth]{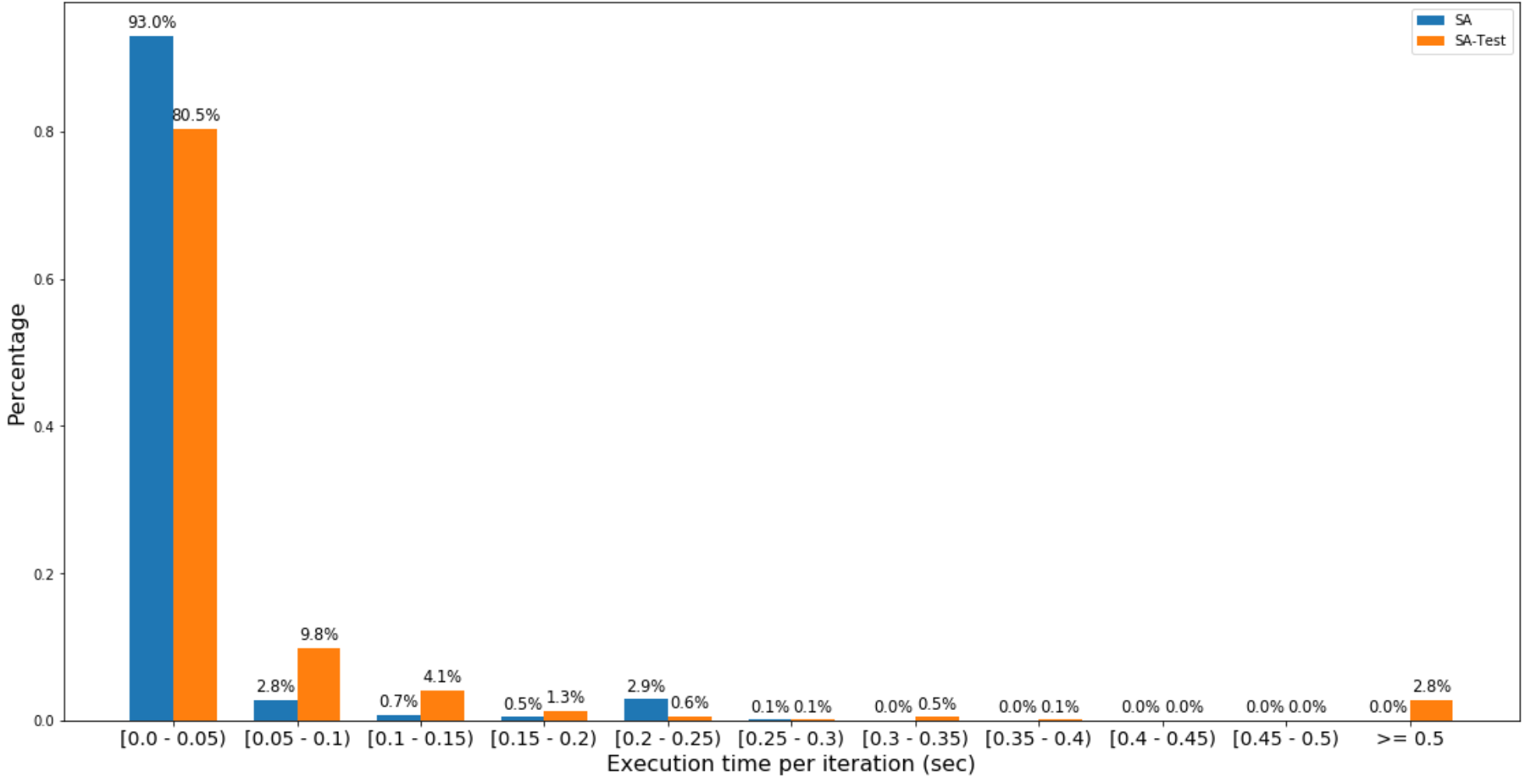}
\end{center}
\caption{Execution time performance profile for SA with no solution improvement}
\label{Fig:SA_Time_Profile}
\end{figure}

As expected, the execution times when not using the solution improvement routine are lower, since there are fewer steps to complete at each iteration, but the distribution of the execution time when using solution improvement is not much worse. In both cases, the majority of the instances have an average execution time below 50 milliseconds. 

The execution time per iteration, in both cases of the simulated annealing framework, are in the same order of magnitude as the execution times of all the studied algorithms in \cite{watel2016practical}. The main difference is that the total execution time of the simulated annealing algorithm is larger given the number of iterations to perform. In any case, the vast majority of the instances solve within seconds, or a few minutes, which makes this approach appealing to use given the better results in solution quality. Moreover, we can always use the solution given by the best benchmark algorithm, which takes a few milliseconds to solve, as the initial laminar family and start the simulated annealing framework from there.

\section{Conclusions and future work}\label{Sec:Conclusions}

We developed a simulated annealing framework to solve the directed Steiner tree problem based on the approach proposed in \cite{siebert2019linear}. We gave an efficient algorithm to solve the subproblem of each tree structure and then we used simulated annealing to find better solutions. We compared the proposed framework with the algorithms studied in \cite{watel2016practical}, and we concluded that our approach outperforms these algorithms in solution quality.

Future research might be directed in applying the insights obtained in this paper to other solution techniques for the problem. For instance, we can use our algorithm to solve each laminar family subproblem to find better upper bounds in a Branch and Bound setting. In particular, we wonder if we can obtain a significant reduction in execution times when we use the Branch and Ascent approach proposed in \cite{de2001dual}. At any node of the search tree, we can get candidate laminar families from the support of the fractional solution, and then solve the problem for that set of tree structures to compute better primal bounds.

\section*{Acknowledgement}
This research was partially supported by Office on Naval Research grants N00014-15-1-2078 and N00014-18-1-2075 to the Georgia Institute of Technology, and by the CONICYT (Chilean National Commission for Scientific and Technological Research) through the Doctoral Fellowship program ``Becas Chile'', Grant No. 72160393
\newpage


	\setlength\bibitemsep{\baselineskip}  
	\printbibliography[title={References}]

@article{bertsimas1993simulated,
	Author = {Bertsimas, Dimitris and Tsitsiklis, John and others},
	Journal = {Statistical science},
	Number = {1},
	Pages = {10--15},
	Publisher = {Institute of Mathematical Statistics},
	Title = {Simulated annealing},
	Volume = {8},
	Year = {1993}}

@article{kirkpatrick1983optimization,
	Author = {Kirkpatrick, Scott and Gelatt, C Daniel and Vecchi, Mario P},
	Journal = {science},
	Number = {4598},
	Pages = {671--680},
	Publisher = {American Association for the Advancement of Science},
	Title = {Optimization by simulated annealing},
	Volume = {220},
	Year = {1983}}

@article{aarts1988simulated,
	Author = {Aarts, Emile and Korst, Jan},
	Publisher = {New York, NY; John Wiley and Sons Inc.},
	Title = {Simulated annealing and Boltzmann machines},
	Year = {1988}}

@incollection{van1987simulated,
	Author = {Van Laarhoven, Peter JM and Aarts, Emile HL},
	Booktitle = {Simulated annealing: Theory and applications},
	Pages = {7--15},
	Publisher = {Springer},
	Title = {Simulated annealing},
	Year = {1987}}

@article{song2003combinatorics,
	Author = {Song, Yun S},
	Journal = {Annals of Combinatorics},
	Number = {3},
	Pages = {365--379},
	Publisher = {Springer},
	Title = {On the combinatorics of rooted binary phylogenetic trees},
	Volume = {7},
	Year = {2003}}

@article{bordewich2005computational,
	Author = {Bordewich, Magnus and Semple, Charles},
	Journal = {Annals of combinatorics},
	Number = {4},
	Pages = {409--423},
	Publisher = {Springer},
	Title = {On the computational complexity of the rooted subtree prune and regraft distance},
	Volume = {8},
	Year = {2005}}

@article{takahashi1990approximate,
	Author = {Takahashi, Hiromitsu},
	Journal = {Math. Japonica.},
	Pages = {573--577},
	Title = {An approximate solution for the Steiner problem in graphs},
	Volume = {6},
	Year = {1990}}

@inproceedings{czech2002parallel,
	Author = {Czech, Zbigniew J and Czarnas, Piotr},
	Booktitle = {Proceedings 10th Euromicro workshop on parallel, distributed and network-based processing},
	Organization = {IEEE},
	Pages = {376--383},
	Title = {Parallel simulated annealing for the vehicle routing problem with time windows},
	Year = {2002}}

@article{ram1996parallel,
	Author = {Ram, D Janaki and Sreenivas, TH and Subramaniam, K Ganapathy},
	Journal = {Journal of parallel and distributed computing},
	Number = {2},
	Pages = {207--212},
	Publisher = {Elsevier},
	Title = {Parallel simulated annealing algorithms},
	Volume = {37},
	Year = {1996}}

@article{greening1990parallel,
	Author = {Greening, Daniel R},
	Journal = {Physica D: Nonlinear Phenomena},
	Number = {1-3},
	Pages = {293--306},
	Publisher = {Elsevier},
	Title = {Parallel simulated annealing techniques},
	Volume = {42},
	Year = {1990}}

@article{de2001dual,
	Author = {de Arag{\~a}o, Marcus Poggi and Uchoa, Eduardo and Werneck, Renato F},
	Journal = {Electronic Notes in Discrete Mathematics},
	Pages = {150--153},
	Publisher = {Elsevier},
	Title = {Dual heuristics on the exact solution of large Steiner problems},
	Volume = {7},
	Year = {2001}}

@article{dolan2002benchmarking,
	Author = {Dolan, Elizabeth D and Mor{\'e}, Jorge J},
	Journal = {Mathematical programming},
	Number = {2},
	Pages = {201--213},
	Publisher = {Springer},
	Title = {Benchmarking optimization software with performance profiles},
	Volume = {91},
	Year = {2002}}

@article{xu2005survey,
	Author = {Xu, Rui and Wunsch, Donald C},
	Publisher = {Institute of Electrical and Electronics Engineers (IEEE)},
	Title = {Survey of clustering algorithms},
	Year = {2005}}

@article{ming2006fasterdsp,
	Author = {MING-IHsieh, ERIC and Tsai, M},
	Journal = {Journal of information science and Engineering},
	Pages = {1409--1425},
	Publisher = {Citeseer},
	Title = {FasterDSP: A faster approximation algorithm for directed Steiner tree problem},
	Volume = {22},
	Year = {2006}}

@article{watel2016practical,
	Author = {Watel, Dimitri and Weisser, Marc-Antoine},
	Journal = {Journal of Combinatorial Optimization},
	Number = {4},
	Pages = {1327--1370},
	Publisher = {Springer},
	Title = {A practical greedy approximation for the directed Steiner tree problem},
	Volume = {32},
	Year = {2016}}

@incollection{duin2000preprocessing,
	Author = {Duin, Cees},
	Booktitle = {Advances in Steiner Trees},
	Pages = {175--233},
	Publisher = {Springer},
	Title = {Preprocessing the Steiner problem in graphs},
	Year = {2000}}

@article{polzin2001improved,
	Author = {Polzin, Tobias and Daneshmand, Siavash Vahdati},
	Journal = {Discrete Applied Mathematics},
	Number = {1-3},
	Pages = {263--300},
	Publisher = {Elsevier},
	Title = {Improved algorithms for the Steiner problem in networks},
	Volume = {112},
	Year = {2001}}

@inproceedings{polzin2002extending,
	Author = {Polzin, Tobias and Daneshmand, Siavash Vahdati},
	Booktitle = {European Symposium on Algorithms},
	Organization = {Springer},
	Pages = {795--807},
	Title = {Extending reduction techniques for the Steiner tree problem},
	Year = {2002}}

@article{rehfeldt2015generic,
	Author = {Rehfeldt, Daniel},
	Title = {A generic approach to solving the Steiner tree problem and variants},
	Year = {2015}}

@article{duin1997efficient,
	Author = {Duin, Cees and Vo$\beta$, Stefan},
	Journal = {Networks: An International Journal},
	Number = {2},
	Pages = {89--105},
	Publisher = {Wiley Online Library},
	Title = {Efficient path and vertex exchange in Steiner tree algorithms},
	Volume = {29},
	Year = {1997}}

@inproceedings{halperin2003polylogarithmic,
	Author = {Halperin, Eran and Krauthgamer, Robert},
	Booktitle = {Proceedings of the thirty-fifth annual ACM symposium on Theory of computing},
	Organization = {ACM},
	Pages = {585--594},
	Title = {Polylogarithmic inapproximability},
	Year = {2003}}

@article{zelikovsky199311,
	Author = {Zelikovsky, Alexander Z},
	Journal = {Algorithmica},
	Number = {5},
	Pages = {463--470},
	Publisher = {Springer},
	Title = {An 11/6-approximation algorithm for the network Steiner problem},
	Volume = {9},
	Year = {1993}}

@article{jain1999data,
	Author = {Jain, Anil K and Murty, M Narasimha and Flynn, Patrick J},
	Journal = {ACM computing surveys (CSUR)},
	Number = {3},
	Pages = {264--323},
	Publisher = {Acm},
	Title = {Data clustering: a review},
	Volume = {31},
	Year = {1999}}

@inproceedings{macqueen1967some,
	Author = {MacQueen, James and others},
	Booktitle = {Proceedings of the fifth Berkeley symposium on mathematical statistics and probability},
	Number = {14},
	Organization = {Oakland, CA, USA},
	Pages = {281--297},
	Title = {Some methods for classification and analysis of multivariate observations},
	Volume = {1},
	Year = {1967}}

@book{felsenstein2004inferring,
	Author = {Felenstein, Joseph},
	Publisher = {Sinauer associates Sunderland, MA},
	Title = {Inferring phylogenies},
	Volume = {2},
	Year = {2004}}

@article{bryant2004splits,
	Author = {Bryant, David},
	Journal = {Annals of Combinatorics},
	Number = {1},
	Pages = {1--11},
	Publisher = {Springer},
	Title = {The splits in the neighborhood of a tree},
	Volume = {8},
	Year = {2004}}

@article{siebert2019linear,
  Author =  {Siebert, Matias and Ahmed, Shabbir and Nemhauser, George},
  Journal = {Networks},
  Year = 2019,
  Title =  {A Linear Programming Based Approach to the Steiner Tree Problem with a Fixed Number of Terminals},
  doi = {10.1002/net.21913},
  url = {https://doi.org/10.1002/net.21913}
}

@incollection{koch2001steinlib,
	Author = {Koch, Thorsten and Martin, Alexander and Vo{\ss}, Stefan},
	Booktitle = {Steiner trees in industry},
	Pages = {285--325},
	Publisher = {Springer},
	Title = {SteinLib: An updated library on Steiner tree problems in graphs},
	Year = {2001}}

@article{wong1984dual,
	Author = {Wong, Richard T},
	Journal = {Mathematical programming},
	Number = {3},
	Pages = {271--287},
	Publisher = {Springer},
	Title = {A dual ascent approach for Steiner tree problems on a directed graph},
	Volume = {28},
	Year = {1984}}

@article{robins2005tighter,
	Author = {Robins, Gabriel and Zelikovsky, Alexander},
	Journal = {SIAM Journal on Discrete Mathematics},
	Number = {1},
	Pages = {122--134},
	Publisher = {SIAM},
	Title = {Tighter bounds for graph Steiner tree approximation},
	Volume = {19},
	Year = {2005}}

@article{chlebik2008steiner,
	Author = {Chleb{\'\i}k, Miroslav and Chleb{\'\i}kov{\'a}, Janka},
	Journal = {Theoretical Computer Science},
	Number = {3},
	Pages = {207--214},
	Publisher = {Elsevier},
	Title = {The Steiner tree problem on graphs: Inapproximability results},
	Volume = {406},
	Year = {2008}}

@incollection{karp1972reducibility,
	Author = {Karp, Richard M},
	Booktitle = {Complexity of computer computations},
	Pages = {85--103},
	Publisher = {Springer},
	Title = {Reducibility among combinatorial problems},
	Year = {1972}}

@article{bern1989steiner,
	Author = {Bern, Marshall and Plassmann, Paul},
	Journal = {Information Processing Letters},
	Number = {4},
	Pages = {171--176},
	Publisher = {Elsevier},
	Title = {The Steiner problem with edge lengths 1 and 2},
	Volume = {32},
	Year = {1989}}

@article{dreyfus1971steiner,
	Author = {Dreyfus, Stuart E and Wagner, Robert A},
	Journal = {Networks},
	Number = {3},
	Pages = {195--207},
	Publisher = {Wiley Online Library},
	Title = {The Steiner problem in graphs},
	Volume = {1},
	Year = {1971}}

@article{byrka2013steiner,
	Author = {Byrka, Jaros{\l}aw and Grandoni, Fabrizio and Rothvoss, Thomas and Sanit{\`a}, Laura},
	Journal = {Journal of the ACM (JACM)},
	Number = {1},
	Pages = {6},
	Publisher = {ACM},
	Title = {Steiner tree approximation via iterative randomized rounding},
	Volume = {60},
	Year = {2013}}

@article{charikar1999approximation,
	Author = {Charikar, Moses and Chekuri, Chandra and Cheung, To-yat and Dai, Zuo and Goel, Ashish and Guha, Sudipto and Li, Ming},
	Journal = {Journal of Algorithms},
	Number = {1},
	Pages = {73--91},
	Publisher = {Elsevier},
	Title = {Approximation algorithms for directed Steiner problems},
	Volume = {33},
	Year = {1999}}

\newpage
\begin{appendices}

\addtocontents{toc}{\protect\renewcommand{\protect\cftchappresnum}{\appendixname\space}}
\addtocontents{toc}{\protect\renewcommand{\protect\cftchapnumwidth}{6em}}


\section{Simulated annealing instances}
\label{app:instances}

Table \ref{app:Table:Studied_Instances} contains the information about all the instances studied in Section \ref{Sec:Results}. Table \ref{app:Table:Studied_Instances} contains the name of the instances, as well as the number of nodes, arcs and number of terminal nodes (excluding the root node). Column `OPT' contains the value of an optimal solution to the instance. For instances whose optimal value is not known, we use the symbol `-'. Column `BB' corresponds to the value of the best benchmark solution. Columns `SA 1k' , `SA-Test 1k' , and `SA-Rect 1k' correspond to the values of the simulated annealing (SA), SA with solution improvement, and SA for rectilinear graphs using 1,000 iterations. Columns `SA 5k' , `SA-Test 5k' , and `SA-Rect 5k' correspond to the values of the SA, SA with solution improvement, and SA for rectilinear graphs using 5,000 iterations. For not rectilinear instances, we use the symbol `-' in columns `SA-Rect 1k' and `SA-Rect 5k'.

\tiny
\begin{longtable}{rrrr|r|r|rrr|rrr}
\caption{Studied instances}\label{app:Table:Studied_Instances}\\
 \toprule
	Instance & $|V|$ & $|A|$ & $|R|$ & OPT & BB & SA 1k & SA-Test 1k & SA-Rect 1k & SA 5k & SA-Test 5k & SA-Rect 5k\\ \midrule
	wrp3-11 & 128 & 454 & 10 & 1,100,361 & 1,100,374 & 1,100,361 & 1,100,361 & - & 1,100,361 & 1,100,361 & - \\
wrp3-12 & 84 & 298 & 11 & 1,200,237 & 1,200,237 & 1,200,237 & 1,200,237 & - & 1,200,237 & 1,200,237 & - \\
wrp3-13 & 311 & 1226 & 12 & 1,300,497 & 1,300,524 & 1,300,497 & 1,300,497 & - & 1,300,497 & 1,300,497 & - \\
wrp3-14 & 128 & 494 & 13 & 1,400,250 & 1,400,251 & 1,400,250 & 1,400,250 & - & 1,400,250 & 1,400,250 & - \\
wrp3-15 & 138 & 514 & 14 & 1,500,422 & 1,500,422 & 1,500,422 & 1,500,422 & - & 1,500,422 & 1,500,422 & - \\
wrp3-16 & 204 & 748 & 15 & 1,600,208 & 1,600,224 & 1,600,208 & 1,600,208 & - & 1,600,208 & 1,600,208 & - \\
\bottomrule
\end{longtable}

\normalsize

\end{appendices}

\end{document}